\documentclass[reqno, 11pt]{amsart}

\usepackage[algoruled,lined,noresetcount,norelsize]{algorithm2e}

\usepackage{graphicx}
\usepackage{amsmath,amssymb,amsthm}
\usepackage{mathrsfs}
\usepackage{mathabx}\changenotsign
\usepackage{dsfont}
\usepackage{xcolor}
\usepackage[backref]{hyperref}
\hypersetup{
    colorlinks,
    linkcolor={red!60!black},
    citecolor={green!60!black},
    urlcolor={blue!60!black}
}
\usepackage[T1]{fontenc}
\usepackage{lmodern}
\usepackage[babel]{microtype}
\usepackage[english]{babel}

\usepackage{geometry}
\numberwithin{equation}{section}
\numberwithin{figure}{section}

\usepackage{enumitem}

\theoremstyle{plain}
\newtheorem{thm}{Theorem}[section]

\newtheorem{obs}[thm]{Observation}
\newtheorem{prob}[thm]{Problem}
\newtheorem{clm}[thm]{Claim}

\newtheorem{lemma}[thm]{Lemma}

\theoremstyle{definition}

\newtheorem{dfn}[thm]{Definition}

\newcommand\eps{\varepsilon}
\newcommand{\Exp}{\mathbb{E}}
\newcommand{\Prob}{\mathbb{P}}
\newcommand{\cB}{\mathcal{B}}
\newcommand{\cT}{\mathcal{T}}
\newcommand{\cS}{\mathcal{S}}
\newcommand{\cR}{\mathcal{R}}
\newcommand{\goodx}{\operatorname{GOOD}_x}
\newcommand{\dang}{\operatorname{dang}}

\newcommand{\weight}{\text{weight}}

\renewcommand{\subset}{\subseteq}
\renewcommand{\setminus}{\smallsetminus}

\begin{document}

\title[Walker-Breaker games on $G_{n,p}$]{Walker-Breaker games on $G_{n,p}$}

\author[Dennis Clemens]{Dennis Clemens}
\author[Pranshu Gupta]{Pranshu Gupta}
\author[Yannick Mogge]{Yannick Mogge}

\thanks{This research was done when the second author was affiliated with the Hamburg University of Technology. The research of the first and third author is supported by Deutsche Forschungsgemeinschaft (Project CL 903/1-1). }

\address{Hamburg University of Technology, Institute of Mathematics, Hamburg, Germany.}
\email{dennis.clemens@tuhh.de, yannick.mogge@tuhh.de}
\address{University of Passau, Faculty of Computer Science and Mathematics, Passau, Germany.}
\email{pranshu.gupta@uni-passau.de}

\maketitle

\begin{abstract}
    The Maker-Breaker connectivity game and Hamilton cycle game 
    belong to the best studied games in positional games theory,
    including results on biased games, games on random graphs and
    fast winning strategies. Recently, the Connector-Breaker game variant, in which Connector has to claim edges such that her graph stays connected throughout the game, 
    as well as the Walker-Breaker game variant, 
    in which Walker has to claim her edges according to a walk, have
received growing attention.
	
	For instance, London and Pluh\'ar studied the threshold bias for the Connector-Breaker connectivity game on a complete graph $K_n$, and showed that there is a big difference between the cases when Maker's bias equals $1$ or $2$. Moreover, a recent result by the first and third author as well as Kirsch shows that the threshold probability $p$ for the $(2:2)$ Connector-Breaker connectivity game on a random graph $G\sim G_{n,p}$ is of order $n^{-2/3+o(1)}$. We extent this result further to Walker-Breaker games and prove that this probability is also enough for Walker to create a  Hamilton cycle.
\end{abstract}

\section{Introduction}

A positional game is a perfect information game for two players, which is played on a \emph{board} $X$, equipped with a family $\mathcal{F} \subset 2^X$ of subsets of the board, which represent the \emph{winning sets}. In the \emph{unbiased} case of such a game, during each round both players claim one previously unclaimed element of the board according to some predefined rules. For example, in a $(1:1)$ Maker-Breaker game both players alternately claim one element of $X$. Maker wins the game if, by the time every element of the board is occupied by either player, she has claimed all elements of a winning set; otherwise Breaker wins the game.

If we allow Maker and Breaker to claim up to $m$ and respectively $b$ elements of the board in every round, we call the game \emph{biased}. Note that Maker-Breaker games are \emph{bias monotone}, since claiming more elements is never a disadvantage for any player, see e.g.~\cite{hefetz2014positional}. This also means that for any fixed $m$ and any non-trivial family $\mathcal{F}$, i.e.~every winning set has more than $m$ elements, we can find an integer $b_0=b_0(m,\mathcal{F})$, called {\em threshold bias}, such that Breaker wins the $(m:b)$ Maker-Breaker game if and only if $b \geq b_0$ holds.

\medskip

Most research on positional games in the last decades concentrates on the case where the board $X$ is the edge set of a given graph. Typical choices are the complete graph on $n$ vertices, which we denote with $K_n$, or a random graph $G$ sampled according to the \emph{binomial random graph model} $G_{n,p}$, for which we fix $n$ vertices and every possible edge is included in the graph independently with probability $p$; we abbreviate this by $G \sim G_{n,p}$. 
Note that for this model and for monotone increasing graph properties $\mathcal{F}$, e.g.~being connected or containing specific subgraphs, we always have a \emph{threshold probability} $p^*$ \cite{bollobas1987threshold} such that
$$ \mathbb{P} (G \sim G_{n,p} \text{ satisfies } \mathcal{F}) \rightarrow \begin{cases} 0 & \text{ if } p = o(p^*) \\ 1 & \text{ if } p = \omega(p^*) \end{cases}
$$
when $n$ tends to infinity.
Some properties $\mathcal{F}$ even have a \emph{sharp threshold} in the sense that the following holds:
$$ \mathbb{P} (G \sim G_{n,p} \text{ satisfies } \mathcal{F}) \rightarrow \begin{cases} 0 & \text{ if } p \leq (1+o(1))p^* \\ 1 & \text{ if } p \geq (1+o(1))p^* . \end{cases}
$$

\medskip

{\bf Connectivity and Hamiltonicity game.} Consider the $(1:b)$ Maker-Breaker \emph{connectivity game} and \emph{Hamiltonicity game} played on the edges of some graph $G$, where the winning sets are all spanning subgraphs and Hamilton cycles of $G$, respectively. The connectivity game was first studied by Lehman~\cite{lehman1964solution}, who showed that Maker as the second player wins the $(1:1)$ Maker-Breaker connectivity game on any graph $G$ if and only if $G$ contains two edge-disjoint spanning trees. A natural next step then was to look at the $(1:b)$ version of the connectivity game played on $K_n$. For this, Chv\'atal and Erd\H{o}s~\cite{CE} proved that the threshold bias is bounded from above by $(1 + o(1))n / \ln n$, and this bound trivially carries over to the Hamiltonicity game as well. A matching lower bound for the connectivity game was later proven by Gebauer and Szabó \cite{gebauer2009asymptotic}, and for the Hamiltonicity game by Krivelevich \cite{K11}.

\medskip

Now, let us compare these results with Maker-Breaker games in which both players play randomly. Then, having bias $(m:b)$, Maker's final graph behaves similarly to a random graph $G \sim G_{n,p}$ with $p= m/(m+b)$. It is well known that the (sharp) threshold probability $p^*$ for $G \sim G_{n,p}$ being connected or containing a Hamilton cycle is of size $(1+o(1))\ln n / n$ (see e.g. \cite{Bollobas,janson2011random}). If $m=1$ this corresponds to $b = (1 + o(1))n / \ln n$, which matches the above mentioned threshold biases perfectly. Thus, for almost all values of $b$, the outcome of the $(1:b)$ Maker-Breaker connectivity game and Hamiltonicity game on $K_n$ is most likely the same, no matter if both players play perfectly or just pick their edges at random. Usually this phenomenon is referred to as the \emph{probabilistic intuition}. There is a huge variety of games which fulfil this intuition, for example the perfect matching game, the Hamiltonicity game \cite{K11}, and the $(m:b)$ connectivity game if $m = o(\ln n)$ \cite{hefetz2011doubly}, but there also exist games for which this intuition fails, such as the $H$-game \cite{bednarska2000biased},
the $K_t$-factor game~\cite{liebenau2022threshold} or the diameter game \cite{balogh2009diameter}.

\medskip

Instead of giving Breaker more power by increasing his bias, one can also give him more power by playing unbiased games on a thinner board. This study was initiated by Stojakovi\'{c} and Szabó \cite{milos_tibor}, when they considered Maker-Breaker games played on a random graph $G \sim G_{n,p}$. We say that a graph $G \sim G_{n,p}$ has a property $\mathcal{F}$ asymptotically almost surely (a.a.s.), if $\mathbb{P}(G \text{ has property }\mathcal{F}) \rightarrow 1$ for $n \rightarrow \infty$. Stojakovi\'{c} and Szabó looked at a variety of games and focused on the threshold probability $p^*$ at which an almost sure Breaker's win turns into an almost sure Maker's win. Since the property of Maker having a winning strategy is monotone increasing, the existence of such a threshold is guaranteed as mentioned above. For the connectivity game and the Hamiltonicity game the threshold probability needs to satisfy $p^* \geq (1+o(1)) \ln n / n$, since for smaller $p$ a random graph $G \sim G_{n,p}$ almost surely contains isolated vertices. A matching upper bound for the connectivity game was proven by Stojakovi\'{c} and Szabó, while an analogous bound for the Hamiltonicity game was obtained later in~\cite{hefetz2009sharp}.

\medskip

{\bf Connector-Breaker and Walker-Breaker games.} Another way to increase Breakers power in the game is to put restrictions on the edges that Maker is allowed to claim. Recently, games with such restrictions have received increasing attention under various names, see e.g.~\cite{corsten2020odd,forcan2020walkermaker,mikalavcki2022spanning,london2018spanning}.

As a first type let us consider Connector-Breaker games, which were introduced by London and Pluh\'ar~\cite{london2018spanning} under the name PrimMaker-Breaker games and which have also been discussed in~\cite{corsten2020odd}. In this variant \emph{Connector} (in the role of Maker) needs to pick her edges in such a way, that her graph stays connected throughout the game. London and Pluh\'ar proved the following Connector-Breaker variant of Lehman's theorem for the unbiased Connector-Breaker connectivity game on $K_n$, in which Connector wins if and only if she occupies a spanning tree.

\begin{thm}\label{thm:cb-connectivity}
	Playing the $(1:1)$ Connector-Breaker connectivity game on a graph $G$ with $n$ vertices, Connector has a winning strategy as the first player if and only if $G$ contains a copy of $H_n$, where $H_n$ is obtained from the complete bipartite graph $K_{n-2,2}$ by putting an additional edge inside its two-element color class.
\end{thm}

From the above result it immediately follows that for Maker to even have a chance in the unbiased game on $G \sim G_{n,p}$, the graph has to be extremely dense. In particular, the threshold probability $p^\ast$ is close to $1$. Furthermore, one can easily see that for any $b \geq 2$ the $(1:b)$ Connector-Breaker connectivity game is won by Breaker on every graph $G$. Hence, the mentioned restriction leads to a big difference in the outcome of a perfectly played game when compared to the usual Maker-Breaker games.

Based on the above facts, London and Pluh\'ar~\cite{london2018spanning} questioned whether similar differences can still be observed if Maker has a bias larger than 1. For the $(2:b)$ Connector-Breaker game on $K_n$ they proved the following.

\begin{thm}\label{thm:cb-connectivity2}
	Let $n$ be large enough. Playing the $(2:b)$ Connector-Breaker connectivity game on $K_n$, Connector wins if $b < n /(8 \ln n)$ and Breaker wins if $b > n / \ln n$.
\end{thm}

Hence, we see that for Connector-Breaker games a small change in Connector's bias can lead to a big difference with respect to the threshold bias. Moreover, Theorem~\ref{thm:cb-connectivity2} shows that the threshold bias in the $(2:b)$ Connector-Breaker variant is of the same order as in the Maker-Breaker variant. 

\smallskip

A natural next question, which has also been asked by London and Pluh\'ar~\cite{london2018spanning}, then is whether the threshold probability $p^\ast$ for a game on $G_{n,p}$ highly depends on the given biases, and moreover, whether for large enough biases the threshold probability $p^\ast$ gets closer to the corresponding threshold for Maker-Breaker games. A partial answer on this
was obtained by Kirsch and the first and third author~\cite{clemens2021connector}. They showed that the threshold probability in this case is of size $n^{-2/3 + o(1)}$, which is very different from the earlier mentioned Maker-Breaker results.

\medskip

Walker-Breaker games further restrict the possible choices for \emph{Walker} (in the role of Maker): She is only allowed to claim edges of the board according to a walk. That is, in each round she must either claim a free edge incident to her current position or walk along an edge that she claimed earlier in the game, hence making its other endpoint the new position of Walker. These games were introduced by Espig, Frieze, Krivelevich, and Pegden \cite{espig2015walker}, and further studied by the first author and Tran \cite{clemens2016creating} as well as by Forcan and Mikala\v{c}ki (see e.g. \cite{forcan2020walkermaker,mikalavcki2022spanning}), amongst others. For the $(2:b)$ Walker-Breaker connectivity and Hamiltonicity game on $K_n$, Forcan and Mikala\v{c}ki were able to show that the threshold bias is again of the order $n / \ln n$, which behaves similarly to the $(2:b)$ Connector-Breaker variant discussed above. 

It then is natural to ask whether the $(2:2)$ Walker-Breaker connectivity on $G_{n,p}$ behaves similarly
to the $(2:2)$ Connector-Breaker connectivity on $G_{n,p}$ as well.
In this paper we will show that this is indeed the case, and extend our result further to give a winning strategy when Connector aims for a Hamilton cycle.

\begin{thm}\label{thm:main}
	Let $\eps\in (0,1)$. Then, for $p\geq n^{-2/3+\eps}$, playing a $(2:2)$ Walker-Breaker game on the edges of a random graph $G\sim G_{n,p}$, Walker a.a.s.~has a strategy to occupy a Hamilton cycle.
\end{thm}

Note that this result is optimal up to the constant $\eps$ in the exponent. Indeed, by the result from~\cite{clemens2021connector} it follows that for 
$p\leq n^{-2/3-\eps}$ a.a.s.~Breaker has a strategy to prevent Connector (and thus the same is true for Walker) from creating any spanning structure. Moreover, for the Maker side we even prove a slightly stronger statement in Section~\ref{sec:strategy} (see Theorem \ref{thm:main.general}) as one of the substrategies in our main strategy uses the approach from~\cite{F15Gen} which allows to create a graph which behaves almost like a typical random graph. Further details can also be found in Section~\ref{sec:prelim}.

\smallskip

\subsection{Organization of the paper.} In Section~\ref{sec:prelim} we introduce some known results from the theory of positional games as well as some basic facts from the probabilistic method. In Section~\ref{sec:structures} we introduce some graph theoretic structures which are needed to describe Walker's strategy and state a Technical Lemma on the distribution of such structures in a random graph. Based on this Technical Lemma, we state Walker's strategy in Section~\ref{sec:strategy} and prove that a.a.s.~it constitutes a winning strategy in the $(2:2)$ game on $G_{n,p}$ for $p\geq n^{-2/3+\eps}$. 
The proof of the Technical Lemma is given in Section~\ref{sec:technical}. Afterwards, we end the paper with some concluding remarks in Section~\ref{sec:concluding}.

\subsection{Notation}
Most notation used in our paper is standard.
For any graph $G$, we use $V(G)$ and $E(G)$ to denote its vertex set and edge set, respectively, and we set $v(G)=|V(G)|$ as well as $e(G)=|E(G)|$. Given $A,B\subseteq V$, 
we write $E_G(A,B)$ for all edges that have one endpoint in $A$ and the other in $B$; also
$e_G(A,B):=|E_G(A,B)|$, $E_G(A):=E_G(A,A)$ and
$e_G(A):=|E_G(A)|$. For any vertex $v\in V$,
we let $N_G(x)$ be the neighbourhood of $x$ in $G$, we set $N_G(x,A):=N_G(x)\cap A$ and we
let $\deg_G(x)=|N_G(x)|$ denote the degree of $x$ in $G$. Moreover, we write $G-v$ for the graph obtained from $G$ by deleting the vertex $v$ and all its incident edges. Whenever it is clear which graph is been looked at, we may omit the subscript $G$ in the above notation.

Assume a Walker-Breaker game is in progress. Then we denote with $W$ and $B$ the
set of edges which have already been claimed by Walker and Breaker, respectively. Moreover, we let $V(W)$ be the set of vertices incident to at least one edge from $W$. An edge which is neither claimed by Walker nor claimed by Breaker is called a free edge; the set of all free edges is denoted $F$. Additionally, we say that an edge is available if it belongs to $F\cup W$.

Given any graph $G$, we write $G_p$ for the random graph model obtained as follows: having the vertex set $V(G)$ fixed, each edge of $G$
is taken to be an edge of $G_p$ randomly and independently with probability $p$.

To simplify some calculations, we use the following notation: Let $f:\mathbb{N}\rightarrow \mathbb{R}$ and $g:\mathbb{R}\times \mathbb{N} \rightarrow \mathbb{R}$ be functions such that $g$ is monotone in the first variable. Then for every $n\in\mathbb{N}$ and $\alpha\in \mathbb{R}$ we write $f(n)=g(\pm \alpha, n)$ to say that
$g(-\alpha,n)\leq f(n) \leq g(+\alpha,n)$. For instance, we would write $f(n)=\ln^{\pm \alpha}(n)$ instead of $\ln^{-\alpha}(n)\leq f(n) \leq \ln^{+\alpha}(n)$.

\section{Preliminaries}\label{sec:prelim}

\subsection{Beck's winning criterion}

One ingredient for Walker's strategy is Beck's winning criterion for biased Maker-Breaker games.

\begin{thm}\label{thm:beck}[Theorem 1 in~\cite{beck1982remarks}]
	Let $a,b\in \mathbb{N}$, and let $(\mathcal{X},\mathcal{F})$ be a hypergraph such that
	$$
	\sum_{F\in \mathcal{F}} (1+b)^{-|F|/a}< \frac{1}{b+1}\, ,
	$$
	then Breaker has a winning strategy for the $(a:b)$ Maker-Breaker game $(\mathcal{X},\mathcal{F})$.
\end{thm}

\subsection{Continuous Box game} 

Another ingredient we use is the Continuous Box game defined as follows.
The game $CBox(b,1;a_1,\ldots,a_n)$ is played with $n$ pairwise disjoint boxes $F_i$, each with positive real weight $a_i$. The game is played between CMaker and CBreaker. In every round, CMaker claims weights from the boxes such that the total sum of the claimed weights is at most $b$, 
while CBreaker solely removes one box from the game (we may say that CBreaker destroys that box). If, during the game, CMaker succeeds in claiming all the weight of a box, she is declared the winner of the game. Otherwise, i.e. when CBreaker succeeds in destroying all boxes, Breaker wins. Let $\mathcal{S}$ be a strategy for CBreaker where he always destroys a box in which CMaker claims the largest weight.
The following lemma is an easy consequence of the results from~\cite{Hefetz11continuousbox}.

\begin{lemma}\label{lem:Cboxgame}
	Let CMaker and CBreaker play the game 
	$Box(b,1;a_1,\ldots,a_n)$ with boxes of size 
	$a_i$. Then following the strategy $\mathcal{S}$, 
	CBreaker can ensure that the following holds 
	throughout the game: 
	If $F_i$ is a box which is still not destroyed by CBreaker,
	then the weight claimed by CMaker in box $F_i$ is 
	at most $b(\ln n + 1)$.  
\end{lemma}

\subsection{MinBox game}
Our last ingredient for Walker's strategy is the MinBox game, a variant of the Box Game~\cite{CE}, which was introduced in \cite{F15Gen} and motivated by~\cite{gebauer2009asymptotic}. Let positive integers $n,D,b$ and a real $\alpha\in (0,1)$ be given. The MinBox$(n,D,\alpha,b)$ is a $(1:b)$ Maker-Breaker game played on a
family of $n$ pairwise disjoint boxes $F_1,\ldots,F_n$, each of size at least $D$, where Maker wins if and only if she occupies at least an $\alpha$-fraction of elements in each of the $n$ boxes.
To analyse the game, let us use the following definitions throughout the game: For each box $F$, let $w_M(F)$ and $w_B(F)$ denote the number of elements that Maker and Breaker have claimed in $F$, respectively, and set $\text{dang}(F) := w_B(F) - b \cdot w_M(F)$. Call $F$ free if every element of $F$ is not claimed yet, and call it active if $w_M(F)<\alpha|F|$. In \cite{F15Gen} the following was proven.

\begin{thm}[Theorem~2.3 in~\cite{F15Gen}]\label{theorem:MinBox}
	Let $n,b,D\in\mathbb{N}$ and $\alpha\in (0,1)$. 
	Assume that in the game MinBox$(n,D,\alpha,b)$ Maker plays as follows: 
	In each turn, Maker chooses an arbitrary free active box $F$ the danger of which is largest, and then she claims a free element of $F$.
	Then, throughout the game $\text{dang}(F) \leq b(\ln n + 1)$
	holds for every active box $F$.
\end{thm}

\subsection{Probabilistic tools}

Throughout our probabilistic arguments we make often use of the following Chernoff inequalities (see e.g.~\cite{janson2011random}) that help to verify that a given binomial random variable $X \sim \operatorname{Bin}(n,p)$, where each of $n$ independent rounds has probability $p$ of being successful, is typically close to its expectation $\mathbb{E}(X)=np$.

\begin{lemma}\label{lem:Chernoff1}
	If $X \sim \operatorname{Bin}(n,p)$, then
	\begin{itemize}
		\item $\Prob(X<(1-\delta)np)< \exp\left(-\frac{\delta^2np}{2}\right)$ for every $\delta>0$, and
		\item $\Prob(X>(1+\delta)np)< \exp\left(-\frac{np}{3}\right)$ for every $\delta\geq 1$.
	\end{itemize}
\end{lemma}

\begin{lemma}\label{lem:Chernoff2}
	If $X \sim \operatorname{Bin}(n,p)$ and $k\geq 7 \Exp(X)$, then $\Prob(X\geq k)\leq \exp\left(-k\right)$.
\end{lemma}

By a standard application of Chernoff's inequality the following bound on the degree in $G_{n,p}$ can be found.

\begin{clm}\label{clm:degrees}
	Let $\eps\in (0,1)$, $p\geq n^{-2/3}$ and $G\sim G_{n,p}$. Then a.a.s.~$d_G(v)=(1\pm \eps)pn$ for every vertex $v\in V(G)$.
\end{clm}

Since our strategy in Section~\ref{sec:strategy} also uses a randomized substrategy similar to~\cite{F15Gen}, we are able to relate the strategy to local resilience properties of random graphs. For this, we use the following definition. More details follow with Theorem~\ref{thm:main.general} in Section~\ref{sec:strategy}.

\begin{dfn}
	Let $\mathcal{P}=\mathcal{P}(n)$ be a monotone increasing graph property and $\varepsilon,p\in (0,1)$. Then we say that
	$\mathcal{P}$ is $(p,\varepsilon)$-resilient if a random graph $G\sim G_{n,p}$ a.a.s.~satisfies the following:
	For every subgraph $G'\subseteq G$ such that $d_{G'}(v)\leq \varepsilon d_G(v)$ holds for every $v\in V(G)$, it is true that $G\setminus G'\in \mathcal{P}$.
\end{dfn}

Later, Theorem~\ref{thm:main} follows from the more general Theorem~\ref{thm:main.general} together with the following theorem of Lee and Sudakov~\cite{lee2012dirac}.

\begin{thm}[Theorem 1.1 in \cite{lee2012dirac}]
	For $n\in\mathbb{N}$, let $\mathcal{H}=\mathcal{H}_n$ denote the property of containing a Hamilton cycle (on $n$ vertices). Then for every $\varepsilon \in (0,1)$ there exists $C=C(\varepsilon)$ such that the following is true: if $p\geq \frac{C\ln (n)}{n}$, then $\mathcal{H}$ is $(p,\frac{1}{2} - \varepsilon)$-resilient.
\end{thm}

\section{Good structures in $G_{n,p}$}\label{sec:structures}

\subsection{Structures $\mathcal{S}_k$}\label{sec:structures_notation}

During the gameplay, according to the strategy in Section~\ref{sec:strategy}, Walker is often confronted with the following situation: her current position is some vertex $a\in V(G)$ and, using some potential function argument, she decides for a vertex $x\in V(G)$ that she wants to reach next. In order to guarantee that Walker can indeed do so, we make use of copies of a \textit{good structure} $\mathcal{S}_k$ which is defined in the following. Later, Lemma~\ref{lemma:technical} is used to show that Walker can always find suitable copies of that structure in $G$ and hence succeed with her strategy, even if Breaker has already occupied many edges in the game. Before we state this lemma, we however need to introduce some appropriate definitions and collect some basic properties of $\mathcal{S}_k$. 

\begin{dfn}[Good structures]
	Given any positive integer $k$, let $\cT_k$ denote a perfect $3$-ary tree $\cT_k$ of depth $k$. Starting from $\mathcal{T}_k$, the good structure $\cS_k$ is created as follows: subdivide every edge of $\mathcal{T}_k$ with one vertex, and afterwards identify all of its leaves to a single vertex.
\end{dfn}

The left side of Figure~\ref{fig:structure.children} shows $\mathcal{S}_3$.
Whenever we want to refer to the graphs $\mathcal{T}_k$ and $\mathcal{S}_k$, we make use of the following labelling. For the graph $\mathcal{T}_k$, the root is denoted by $s_{k,1}$ and for every $\ell\in [k]$ and $i\in [3^{k-\ell}]$, we let the three children of $s_{\ell,i}$ be $s_{\ell-1,3i-2}$, $s_{\ell-1,3i-1}$ and $s_{\ell-1,3i}$. 
For $\mathcal{S}_k$ we keep this labelling, and we set $s_0:=s_{0,i}$ for every $i\in [3^k]$ to be the identifying vertex. When talking about the subdividing vertices, we write $s^\ast_{\ell-1,3i-j}$ for the middle vertex of the path between vertices $s_{\ell,i}$ and $s_{\ell-1,3i-j}$, for every $\ell\in [k]$, $i\in [3^{k-\ell}]$ and $j\in \{0,1,2\}$. See also the right side of Figure~\ref{fig:structure.children}.

Moreover, in light of the strategy described in Lemma~\ref{lemma:structure_strategy}, we sometimes say that
$\mathcal{S}_k$ \textit{starts} in $s_{k,1}$ and \textit{ends} in $s_0$.
Furthermore, for every $\ell\in [k]$ we call $L_{\ell}:=\{s_{\ell,i}:i\in [3^{k-\ell}]\}$ the \textit{main level} $\ell$, and we use
$I(k)=\{ (\ell,i):~ \ell\in [k], i\in [3^{k-\ell}]\}$ for the set of indices over all vertices belonging to main levels.
Similarly, for every $\ell\in \{0\}\cup [k-1]$ we call $L_{\ell}^\ast:=\{s_{\ell,i}^\ast:i\in [3^{k-\ell}]\}$ the \textit{secondary level} $\ell$, and we use $I^\ast(k)=\{ (\ell,i):~ \ell\in [k]\cup\{0\}, i\in [3^{k-\ell}]\}$ for the set of indices over all vertices belonging to secondary levels. Note that the following holds.

\begin{figure}
	\centering
	\includegraphics[scale=0.8,page=2]{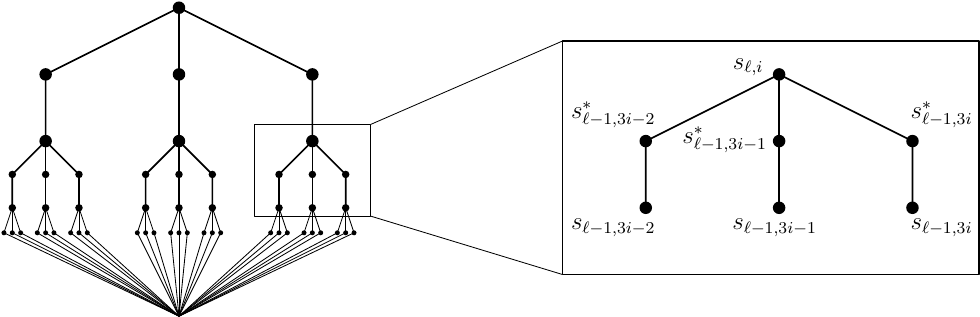}
	\caption{The left part of the picture shows the structure $\mathcal{S}_3$, while the right part depicts the notation of vertices in a small subtree between main levels $\ell$ and $\ell-1$.}
	\label{fig:structure.children}
\end{figure}

\begin{clm}\label{claim:structure_size}
	For every positive integer $k$ we have
	$v(\mathcal{S}_k) = 2\cdot 3^{k}- 1$, $ 
	e(\mathcal{S}_k) = 3^{k+1} - 3,$
	and	the tree $\mathcal{S}_k - s_0$ has $3^{k}$ leaves.
\end{clm}

\begin{proof}
	By construction, $\mathcal{S}_k-s_0$ is a tree with
	$|L_0^\ast|=3^k$ leaves and with
	$$
	v(\mathcal{S}_k-s_0) = \sum_{\ell=1}^k |L_{\ell}| + \sum_{\ell=0}^{k-1} |L_{\ell}^\ast|
	= \sum_{\ell=1}^k 3^{k-\ell} + \sum_{\ell=0}^{k-1} 3^{k-\ell}
	= 2\cdot 3^k - 2\, .
	$$
	In particular, $v(\mathcal{S}_k) = 2\cdot 3^k - 1$. The edge set of
	$\mathcal{S}_k$ consists of all edges of the tree $\mathcal{S}_k-s_0$ and all edges between $s_0$ and $L_0^\ast$. Hence,
	$
	e(\mathcal{S}_k) = v(\mathcal{S}_k-s_0) - 1 + |L_0^\ast|
	= 2\cdot 3^k - 2 - 1 + 3^k = 3^{k+1} - 3
	$.
\end{proof}

The next lemma indicates how Walker can use copies of $\mathcal{S}_k$ in her strategy.

\begin{lemma}\label{lemma:structure_strategy}
	Let $k\geq 1$, and consider a $(2:2)$ Walker-Breaker game on the structure $\cS_k$ with Breaker being the first player and Walker's starting position being $s_{k,1}$. Then Walker has a strategy $S_{\text{structure}}$ to reach the vertex $s_0$ within $k$ rounds.
\end{lemma}

\begin{proof}
	We prove this lemma by induction on $k$. For $k=1$, the structure $\mathcal{S}_k$ consists of three edge-disjoint paths of length 2 between $s_{1,1}$ and $s_0$, and by assumption $s_{1,1}$ is Walker's starting position. Since Breaker can only block two of these paths within the first round, Walker can find a free path with which she reaches $s_0$ within her first move.
	
	For $k>1$ consider the structure $\cS_k$ starting in $s_{k,1}$ and ending in $s_0$. This structure consists of three copies of the structure $\cS_{k-1}$ starting in one of the vertices $s_{k-1,i}$, $i \in [3]$, and ending in $s_0$, and the three paths of the form $s_{k,1}s^*_{k-1,i}s_{k-1,i}$. After Breaker's first move, at least one of the paths $s_{k,1}s^*_{k-1,i}s_{k-1,i}$ as well as the copy of $\cS_{k-1}$ starting from $s_{k-1,i}$ still do not contain any edge claimed by Breaker. Walker claims this path and thus reaches $s_{k-1,i}$ within one move. According to our induction hypothesis Walker can then reach $s_0$ starting from $s_{k-1,i}$ within $k-1$ further moves and thus has a strategy to reach $s_0$ within $k$ moves.
\end{proof}

\subsection{Finding copies of $\mathcal{S}_k$ in $G_{n,p}$}

By Claim~\ref{claim:structure_size} the density
of $\mathcal{S}_k$ is slightly below $3/2$, and hence, for $p\geq n^{-2/3}$ we know that with high probability a random graph
$G\sim G_{n,p}$ contains copies of $\mathcal{S}_k$. (See e.g.~\cite{Bollobas} for subgraph containment in $G_{n,p}$.) Moreover, increasing slightly to $p \geq n^{-2/3+\eps}$ and having that $k=k(\eps)$ is a sufficiently large constant, it can be verified that in expectation
for any two vertices $v,x\in V(G)$ there is a copy of $\mathcal{S}_k$ such that $v$ is the copy of $s_{k,1}$ and $x$ is the copy of $s_0$. Now, throughout the game, Walker aims to find such copies which additionally are free of Breaker's edges and which can be reached from her current position $a$ within one round. 
In order to guarantee that Walker can indeed find such copies, we want to make sure that there exists a collection of many copies of $\mathcal{S}_k$  with the additional property that any Breaker edge can only belong to a comparatively small number of these copies. This property may help later to ensure that Walker can play without facing a situation in which she wants to add a new vertex $x$ to her graph but cannot find a good structure with which she can still reach $x$.

In order to prove the existence of the desired collections of structures we make use of a recursive approach. As we want to keep independence whenever needed in the probabilistic analysis, we initially split the vertex set of $G\sim G_{n,p}$ into several blocks $B_t(s)$, for every $s\in V(\mathcal{S}_k-s_0)$ and $t\in [2]$, and some left-over $R$; then  we aim for copies of $\mathcal{S}_k$ for which the copy of any vertex $s\in V(\mathcal{S}_k-s_0)$ is an element of one of the corresponding blocks $B_t(s)$. In order to make our statements precise, we use the following definitions.

\begin{dfn}\label{dfn:structures.appearance}
	Let a positive integer $k$, graph $G$, and a family 
	$\cB:=\left\{ B(s):~ s\in V(\cS_k-s_0) \right\}$
	of pairwise disjoint subsets of $V(G)$ be given. Furthermore, let $x\in V(G)$, $e\in E(G)$, $v\in B(s_{k,1})$, and $\ell\in [k]$. Then we define the following:
	\begin{itemize}
		\item A copy $S$ of $\cS_k$ in $G$ is called a \textit{$(\cB,x)$-structure}
		if $x$ is the copy of $s_0$ in $S$ and if for every $s\in V(\cS_k-s_0)$ there is a vertex in $B(s)$ which is the copy of $s$ in $S$.
		\item We say that $e$ \textit{sees the vertex} $x$ with respect to $\cB$ if there exists a $(\cB,x)$-structure containing $e$.
		\item We say that $e$ \textit{is relevant for} $v$ with respect to $(\cB,x)$ if there exists a $(\cB,x)$-structure containing both $v$ and $e$.
		\item We say that $e$ \textit{appears between levels $\ell-1$ and $\ell$} with respect to $\mathcal{B}$ if there exists a vertex $s\in L_{\ell -1}^\ast$ such that $e$ is incident with a vertex in $B(s)$, and if $\ell$ is the smallest integer with this property.
		\item We say that $e$ \textit{appears below level $\ell$} with respect to $\mathcal{B}$ if for some $\ell'\leq \ell$ the edge $e$ 
		appears between levels $\ell'-1$ and $\ell'$.
	\end{itemize}
\end{dfn}

Now, we can state our main technical lemma which the strategy of Walker builds on. In order to avoid rounding signs in its proof we may restrict the possible choices of $\eps$. More precisely, we let $\cR$ be the set of real numbers $\eps\in (0,1)$ such that $\log_3(2\eps^{-1}+12)-2$ is a positive integer. Note that, by monotonicity and since $\inf(\cR)=0$, it is enough to prove Theorem~\ref{thm:main} for $\eps\in\cR$.

\begin{figure}
	\centering
	\includegraphics[scale=0.8,page=3]{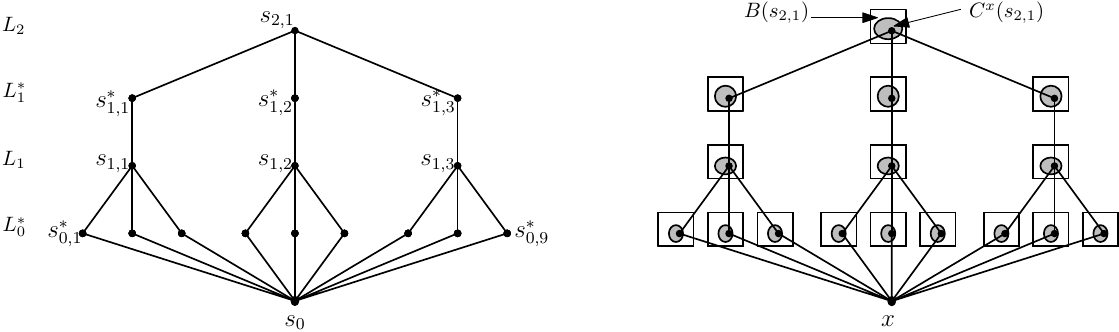}
	\caption{The left picture shows the structure $\mathcal{S}_2$. The right picture shows a $(\mathcal{B},x)$-structure with respect to blocks $B(s)$ depicted by rectangles, and candidate sets $C^x(s)$ depicted by grey areas.}
	\label{fig:boxes.candidates}
\end{figure}

\begin{lemma}[Main Technical Lemma]\label{lemma:technical}
	For every $\eps\in \cR$ there exists a positive integer $k$ such that the following holds. Let the structure $\cS_k$ be given with the labelling from Subsection~\ref{sec:structures_notation}.
	Further, let $p\geq n^{-2/3+\eps}$ and $G\sim G_{n,p}$. Then a.a.s.~there exist a vertex $a\in V(G)$, 
	a partition $V(G)\setminus \{a\}=V_1\cup V_2$,
	families $\mathcal{B}_t=\left\{B_t(s):~ s\in V(\cS_k-s_0)\right\}$ of pairwise disjoint subsets of $V_{t}$ for every $t\in [2]$, 
	and a set $R\subseteq V\setminus (\bigcup_{s\in V(\mathcal{S}_k-s_0)} (B_1(s)\cup B_2(s)) \cup \{a\})$ such that the following properties hold:
	\begin{enumerate}
		\item[(S)] \underline{Sizes:} $|R|\geq \frac{n}{2}$, $|B_t(s)|=\frac{n}{3^{k+10}}$ for every $s\in V(\mathcal{S}_k-s_0)$ and $t\in [2]$,
		and $|N(a,R)|\geq n^{1/3+\eps/2}$.
		\item[(C)] \underline{Candidate sets:} Let $t\in [2]$. For all $x\in V_{3-t}$ there exist candidate sets
		$C^x(s)\subseteq B_t(s)$ for every $s\in V(\cS_k-s_0)$ such that:
		
		\smallskip
		
		\begin{enumerate}
			\item[(C1)] \underline{Number of candidates:} $|C^x(s)|=n^{(3^{\ell+1}-3)\eps} \ln^{\pm 3^{3\ell}}(n)$ for every $\ell\in [k]$ and $s\in L_{\ell}$. 
			\item[(C2)] \underline{Neighbourhoods for main levels:} For every $(\ell,i)\in I(k)$, every vertex $v\in C^x(s_{\ell,i})$ has a neighbour in each of the sets $C^x(s^\ast_{\ell-1,j})$ with $3i-2\leq j\leq 3i$.
			\item[(C3)] \underline{Neighbourhoods for secondary levels:} For every $(\ell,i)\in I^\ast(k)$, every vertex $v\in C^x(s^\ast_{\ell,i})$ has a neighbour in $C^x(s_{\ell,i})$, where we set $C^x(s_{0,i})=\{x\}$.
		\end{enumerate}
		
		\noindent
		Set $C^x:=C^x(s_{k,1})$ from now on.
		
		\smallskip
		
		\item[(E)] \underline{Edge appearances:} 
		Let $t\in [2]$. For every edge $e\in E(G)$ the following holds:
		\begin{enumerate}
			\item[(E1)] If $e$ appears below level $k-1$ with respect to $\mathcal{B}_t$, then $e$ sees at most $\ln^2(n)$ vertices $x\in V_{3-t}$ with respect to $\cB_t$.
			\item[(E2)] If $e$ appears between levels $k-1$ and $k$ with respect to $\mathcal{B}_t$, then $e$ sees at most $n^{1/3+0.1\eps}$ vertices $x\in V_{3-t}$ with respect to $\cB_t$.
		\end{enumerate}

		\item[(R)] \underline{Relevance of edges:} Let $t\in [2]$ and $x\in V_{3-t}$. For any set $Z_1$ of edges appearing below level $k-1$ with respect to $\mathcal{B}_t$,
		and any set $Z_2$ of edges appearing between levels $k-1$ and $k$ with respect to $\mathcal{B}_t$, let $C^x[Z_1,Z_2]$
		denote the vertices in $C^x$ for which no edge of $Z_1\cup Z_2$ is relevant with respect to $(\cB_t,x)$. Then the following holds:
		If $|Z_1|=\ln^4(n)$ and $|Z_2|=n^{1/3+\eps/2}$, and if $A\subseteq N(a,R)$ has size $|A|=n^{1/3}$, then
		$
		e_{G}\big(A,C^x[Z_1,Z_2]\big) \geq n^{1/3+1.5\eps}\, . 
		$
	\end{enumerate}
\end{lemma}

We postpone the proof of Lemma~\ref{lemma:technical} to Section~\ref{sec:technical}. Nevertheless, let us briefly give a reason why we care about the above properties. Set $B(s)=B_1(s)\cup B_2(s)$ for every $s\in V(\mathcal{S}_k)-s_0$ and $\mathcal{B}=\{B(s):~s\in V(\mathcal{S}_k)-s_0\}$. Property (C) promises that for every $x\in V(G)\setminus \{a\}$ we can find a collection of $(\mathcal{B},x)$-structures starting in any vertex of $C^x(s_{k,1})$. 

\begin{obs}\label{obs:CgivesSk}
	Let $a\in V(G)$, a partition $V(G)\setminus \{a\}=V_1\cup V_2$,
	families $\mathcal{B}_t$, and a set $R$ be given according to Lemma~\ref{lemma:technical} such that property (C) holds. Then for every $t\in [2]$, $x\in V_{3-t}$, and $v\in C^x(s_{k,1})\cap B_t(s_{k,1})$ there exists a $(\mathcal{B}_t,x)$-structure starting in $v$ and ending in $x$.
\end{obs}

Indeed, let $x\in V_{3-t}$, then fixing any vertex $v_{k,1}:=v\in C^x(s_{k,1})\cap B_t(s_{k,1})$ we can find a copy of $\mathcal{S}_k$ starting in $v_{k,1}$ and ending in $x$ as follows: By (C2) we know that $v_{k,1}$ has neighbours in each of the sets $C^x(s_{k-1,i}^\ast)\subset B_t(s_{k-1,i}^\ast)$. Picking one such vertex $v_{k-1,i}^\ast$ from each of these sets, we have fixed copies of the vertices $s_{k-1,i}^\ast$ of $\mathcal{S}_k$. Then, by (C3) we know that each $v_{k-1,i}^\ast$ has a neighbour in
$C^x(s_{k-1,i})\subset B_t(s_{k-1,i})$; hence we can pick copies of the vertices $s_{k-1,i}$. This process continues, always switching between (C2) and (C3)
until we reach $C^x(s_{0,i})=\{x\}$ so that $x$ becomes the copy of $s_0$ and a $(\mathcal{B}_t,x)$-structure is found.

Note that (C1) ensures that we have many possible candidates in $C^x(s_{k,1})$ to start with, which in turn means that we have many such structures. 
Additionally, property (E) helps when we want to show that in each move, Breaker can only block a reasonably small number of $(\mathcal{B},x)$-structures. Moreover, property (R) comes in handy  when we want to ensure that Walker, being at her starting position $a$, can always reach some $(\mathcal{B},x)$-structure within one round. More precisely, the sets $Z_1$ and $Z_2$ mentioned in the property represent edges claimed by Breaker, and the set $C^x[Z_1,Z_2]$ represents those candidates in $C^x$
from which Walker can still reach the vertex $x$ by following the strategy from Lemma~\ref{lemma:structure_strategy}. 
Having that $e_{G}\big(A,C^x[Z_1,Z_2]\big)$ is large,
we can ensure that Walker has enough options to reach one of the mentioned candidates within one round. More details are given in Section~\ref{sec:strategy} where we first exhibit a partially randomized strategy for Walker and then prove that it lets Walker win with high probability.

\section{Walker's strategy}\label{sec:strategy}

In this section, we prove Theorem~\ref{thm:main} by showing the following more general statement.

\begin{thm}\label{thm:main.general}
	Let $\eps\in (0,1)$, $p\geq n^{-2/3+\eps}$, and $\mathcal{P}=\mathcal{P}_n$ be a monotone increasing graph property that is $(p,\eps)$-resilient. Then,
	playing a $(2:2)$ Walker-Breaker game on the edges of a random graph $G\sim G_{n,p}$, Walker a.a.s.~has a strategy to occupy a graph with property $\mathcal{P}$.
\end{thm}

By monotonicity, we can restrict our attention to $\eps\in \mathcal{R}$, meaning that $\log_3(2\eps^{-1}+12)-2$ is a positive integer.
We hence can apply Lemma~\ref{lemma:technical} to obtain some output $k\in\mathbb{N}$. For the graph $G\sim G_{n,p}$, we thus can always condition on the properties promised by Lemma~\ref{lemma:technical} and Claim~\ref{clm:degrees}. That is, before the game starts we fix a vertex $a\in V(G)$, a set $R$, and families $\cB_t$ with $t\in [2]$ as described in the lemma. In particular, we assume that all of the properties (S), (C), (E), and (R) hold, which amongst other things provides us with families of $(\mathcal{B}_t,x)$-structures as mentioned in Observation~\ref{obs:CgivesSk}. Moreover, for notational reasons, we set $B(s)=B_1(s)\cup B_2(s)$ for every $s\in V(\mathcal{S}_k)-s_0$ and $\mathcal{B}=\{B(s):~s\in V(\mathcal{S}_k-s_0)\}$. Whenever needed, we assume that $n$ is large enough.

\medskip

In the following we first describe Walker's strategy which combines a deterministic strategy with randomized moves. More precisely, in
Subsections~\ref{subsec:setup} and~\ref{sec:substrategy} we first describe the overall idea, introduce necessary notation and give a substrategy that is used later on. Then, in Subsection~\ref{subsec:strategy}, we describe the full strategy of Walker.
Finally, in Subsection~\ref{subsec:discussion} we show that
Walker can always follow the described strategy and that, against any strategy of Breaker, Walker a.a.s.~manages to occupy a graph with property $\mathcal{P}$. It then follows that Breaker cannot have a winning strategy, i.e.~a strategy which always prevents Walker from occupying a graph with property $\mathcal{P}$, 
and hence Walker must have a deterministic strategy to win the game (see Zermelo's Lemma, e.g.~\cite{beck2008combinatorial}).

\subsection{Our setup}\label{subsec:setup}
In the strategy which we will make precise in Subsection~\ref{subsec:strategy}, Walker alternates between three different sequences of moves, denoted by Sequence I - III. These different sequences allow us to maintain three different goals: 
\begin{enumerate}
	\item[(I)] occupy suitable paths of length 2 starting at the starting vertex $a$,
	\item[(II)] ensure that for every $x\notin V(W)$ there exist
	$(\mathcal{B},x)$-structures the edges of which are available,
	\item[(III)] occupy a graph which "behaves almost like a random graph" (details below).
\end{enumerate}

We remark at this point that the goals (I) and (II) would be enough to prove that Walker wins the connectivity game. As long as she maintains these goals, the main idea of her strategy is the following: Let $x$ be any vertex that Walker wants to include next to her component. By (II) she may find some $(\mathcal{B},x)$-structure $S_x$, the edges of which are available, and by (I) there may be a path of length 2 with which Walker can reach the top vertex of $S_x$. Once this top vertex is reached, Walker can use the edges of $S_x$ to walk towards $x$ (Lemma~\ref{lemma:structure_strategy}) and hence add it to her component. 

\medskip

Property (I) will be guaranteed by Sequence I of the strategy given in Subsection~\ref{subsec:strategy}. In this, Walker makes sure to create a large star with center $a$, and moreover she makes sure to claim suitable edges starting at the leaves of this star, hence creating paths of length 2. These paths are to be created in such a way that Walker can always reach vertices from certain candidate sets $C^x[Z_1,Z_2]$ as described in (R). More details on this are given by Lemma~\ref{lemma:paths_beck}.

\smallskip

Property (II) will be guaranteed by Sequence II of the strategy given in Subsection~\ref{subsec:strategy}. Roughly speaking, the idea is as follows: For every $x\in V(G)$, having $t\in [2]$ such that $x\in V_{3-t}$, we may consider the set $E_x$ consisting of all edges which see the vertex $x$ with respect to $\cB_t$. Using the Continuous Box Game, Walker then ensures that Breaker does not claim too many elements of $E_x$ as long as $x\notin V(W)$. Hence, we are able to find $(\cB_t,x)$-structures which have not been blocked by Breaker so far. 

\smallskip

Finally, property (III) will be guaranteed by Sequence III of the strategy given in Subsection~\ref{subsec:strategy}. Using that part of the strategy, Walker generates a random graph $H\sim G_{\ln^{-1}(n)}$ while the game is proceeding, i.e.~$H\sim G_{n,q}$ with $q=p\ln^{-1}(n)=\omega(n^{-2/3})$. Using a partially randomized strategy, Walker then ensures that a.a.s.~ 
\begin{equation}\label{eq:resilience_condition}
	d_{W\cap H}(v)\geq (1-\eps)d_H(v)
\end{equation}
holds for every $v\in V(H)$ by the end of the game. Since the desired property $\mathcal{P}$ is $(p,\eps)$-resilient, we then have with high probability that $W\cap H$ must have property $\mathcal{P}$ and thus Walker wins. This part of the strategy is motivated by~\cite{F15Gen} and has similarly been used for Walker-Breaker games in~\cite{clemens2016creating,mikalavcki2022spanning}. In the following we provide a few more details of that strategy adapted to our setting.

\medskip

While the game is going on, Walker tosses a coin on every edge of $G$ independently at random, where the probability of success equals $\ln^{-1}(n)$.
If the coin tossed for an edge $e$ shows a success, Walker adds the edge $e$ to the random graph $H$, and moreover, Walker claims the edge if that is possible.

In order to decide on which edge Walker tosses her coin, she always identifies some exposure vertex $v$ (to be defined later in the strategy description) and makes sure that she can reach the vertex $v$ within a small number of rounds. In order to choose an appropriate exposure vertex,
Walker plays an auxiliary MinBox game in parallel, namely MinBox$(n,4pn,0.5\ln^{-1}(n),16k+28)$. In this simulated $(1:16k+28)$ Maker-Breaker game, in which Walker imagines playing as Maker, we have a box $J_v$ of size $4pn$ for every vertex $v\in V(G)$.

Then, once an exposure vertex $v$ is identified and Walker has reached this vertex, she starts tossing the coin on edges
which are incident with $v$, but only for which she has not tossed a coin  earlier, and she stops when the first success happens.
Of course, it may happen that none of her coin tossing is a success, in which case we say that her move is a failure of type I, and Walker just makes an arbitrary move instead.
It may also happen that Walker has success on an edge which was already claimed by Breaker in an earlier round so that Walker cannot take it. Similarly to the previous case, Walker makes an arbitrary move then and we denote it as failure of type II. 

Analogously to~\cite{F15Gen} our final goal is to prove that with high probability at every vertex $v$ only a relatively small number of failures happen, which then yields \eqref{eq:resilience_condition}. 

For the analysis of such an argument, we say that Walker exposes an edge $e\in E(G)$ if she does a coin tossing on it, and we set $U_v \subseteq N_G(v)$ to be the set of all neighbours $w$ of $v$ for which the edge $vw$ has not been exposed yet. Moreover, we introduce counters $f_I(v)$ and $f_{II}(v)$ in order to keep track on the number of failures of type I and II which involve edges incident with $v$. Initially, we set $f_I(v)=f_{II}(v)=0$ for every vertex $v\in V(G)$. 
In order to ensure that the number of failures of type II does not get too large, we apply Theorem~\ref{theorem:MinBox} for MinBox$(n,4pn,0.5\ln^{-1}(n),16k+28)$. Hereby, $w_B(J_v)$ is made to be related to $d_B(v)$ for every $v\in V(G)$, while $w_M(J_v)$ is made to be related to the number of edges incident with $v$ on which the coin was successful. Initially, $w_B(J_v)=w_M(J_v)=0$ for every $v\in V(G)$. Later we will realize that it is very likely that Walker is done with all coin tosses on the edges incident with a fixed vertex $v$,
before Breaker has claimed too many of these edges. This in turn helps to bound the number of failures of type II.

\subsection{A substrategy (Sequence I)} \label{sec:substrategy}

In this subsection we prepare for Sequence I of our main strategy (Section~\ref{subsec:strategy}) and prove the following lemma.

\begin{lemma}\label{lemma:paths_beck}
	Let $b$ be any positive integer and let $n$ be large enough.
	Assume that a $(2:b)$ Walker-Breaker game on some graph $G$ is in progress and that Walker has already claimed the edges of a star of size $n^{1/3}$, with center $a$ and $A$ being the set of leaves.
	Assume further that there exist (not necessarily disjoint) subsets
	$C_1,\ldots,C_{s}\subseteq V(G)\setminus (A\cup \{a\})$, with $s \leq \exp\left(n^{1/3+\eps}\right)$, such that for each $i\in [s]$
	there exist at least $n^{1/3+1.1\eps}$ available edges between
	$A$ and $C_i$. Then Walker has a strategy $\mathcal{S}_{paths}$ 
	that satisfies the following:
	\begin{itemize}
		\item[(i)] The strategy proceeds in sequences of two moves, which always start and end in the vertex $a$, and in which Walker claims exactly one edge from $E_G(A,\bigcup_{i\in [s]} C_i)$.
		\item[(ii)] The strategy ensures that Walker claims
		at least one element from each of the sets $E_G(A,C_i)$. In particular, as long as Walker plays according to this strategy there must always be an edge in $E_G(A,C_i)$ which is either free or taken by Walker.
	\end{itemize}
\end{lemma}

\begin{proof}
	Assume that at any moment in the game Walker's position is the vertex $a$ and she wants to claim some free edge $e=xy\in E_G(A,\bigcup_{i\in [s]} C_i)$, say with $x\in A$. Since, by assumption of the lemma, Walker has already claimed the edge $ax$,
	she can play a sequence of two moves as follows: first go from $a$ to $y$ via $x$, and secondly return to $a$ by using the same edges. If Walker continues playing like this, her strategy already satisfies (i). Moreover, as Walker can claim one free edge of 
	$E_G(A,\bigcup_{i\in [s]} C_i)$ arbitrarily while Breaker claims at most $2b$ edges in the meantime, Walker can imagine playing an auxiliary $(2b:1)$ Maker-Breaker game on the board $E_G(A,\bigcup_{i\in [s]} C_i)$ with the family of winning sets
	$\mathcal{F}=\{E(A,C_i) : i \in [s]\}$, but taking over the role of Breaker since for (ii) she wants to occupy an edge in each of the winning sets. In order to show that Walker can succeed in doing so, it is enough to check the Beck's Criterion (Theorem~\ref{thm:beck}):
	$$
	\sum_{F\in\mathcal{F}} 2^{-|F|/(2b)} =
	\sum_{i \in [s]} 2^{-e_G(A,C_i)/2b} \leq e^{n^{1/3+\varepsilon}} 2^{-n^{1/3+1.1\varepsilon}/2b} =o(1) \, .
	$$
	This proves the lemma.
\end{proof}

\subsection{Strategy description}\label{subsec:strategy}
Throughout the gameplay, Walker alternates between the following sequences of moves which always start and end in the vertex $a$. To keep track of her star with center $a$ (Sequence I), we let $N_a$ denote a set of vertices $w\in N(a,R)$ for which $aw$ is already a Walker's edge. Initially $N_a=\varnothing$. Moreover, once $|N_a|=n^{1/3}$ holds,
the set is never updated again (even if Walker claims another edge incident with $a$). 

In order to apply the Continuous Box Game (Sequence II),
we define weights for the edges and vertices of $G$ as follows.
For every edge $e\in E(G)$ and $t\in [2]$ we define
\begin{align*}
	\weight_t(e) :=
	\begin{cases}
		\ln^{-2}(n) & ~ \text{if $e$ is an edge below level $k-1$ w.r.t.~$\mathcal{B}_t$} \\
		n^{-1/3-0.1\eps} & ~ \text{if $e$ is an edge between level $k-1$ and $k$ w.r.t.~$\mathcal{B}_t$, and} \\
		0 & ~ \text{otherwise}\, .
	\end{cases}
\end{align*}
Moreover, for every $x\in V(G)\setminus \{a\}$ and $t\in [2]$ such that $x\in V_{3-t}$, we dynamically define
\begin{align*}
	\weight(x) := \sum_{e\in B\text{ sees }x \atop \text{w.r.t.~$\cB_t$}} \weight_t(e)\, 
\end{align*}
where the sum is taken over all edges $e$ that see $x$ with respect to $\cB_t$ and which have been claimed by Breaker already. Initially $\weight(x)=0$ for every $x\in V(G)\setminus \{a\}$.

Finally, in light of property (R) from Lemma~\ref{lemma:technical}, we let $C^x[Z_1,Z_2]$ be the set of vertices in $C^x$ for which no edge of $Z_1$ or $Z_2$  is relevant with respect to $(\cB_t,x)$. Moreover, we set 
$$
\goodx := 
\left\{
(Z_1,Z_2):~
\begin{array}{l}
	\text{$Z_1$ consists of edges below level $k-1$ w.r.t.~$\cB_t$,}\\
	\text{$Z_2$ consists of edges between level $k-1$ and $k$ w.r.t.~$\cB_t$,}\\
	|Z_1|=\ln^4(n), ~ \text{ and } ~ |Z_2|=n^{1/3+\eps/2}
\end{array}
\right\}.
$$ 

We are now ready to give the full description of Walker's strategy.

\medskip

{\bf Sequence I:} 
If $|N_a| = n^{1/3}$ and Walker has an edge between $N_a$ and any set $C^x[Z_1,Z_2]$ with $(Z_1,Z_2)\in \goodx$, she proceeds with Sequence II immediately.
Otherwise, she distinguishes two cases.

\medskip

\underline{Case 1.:} Let $|N_a|<n^{1/3}$. Then Walker proceeds as follows:
\begin{itemize}
	\item[(i)] Walker identifies a vertex $w\in N(a,R)$ such that $aw$ is free.
	\item[(ii)] She walks from $a$ to $w$ (hence claiming $aw$) and back to $a$.
\end{itemize}
$N_a$ is updated by adding $w$ to this set. Walker continues with Sequence II afterwards.

\medskip

\underline{Case 2.:} Let $|N_a|= n^{1/3}$. Walker then makes a sequence of two moves according to strategy $\mathcal{S}_{paths}$ with bias $8k+14$, with $A=N_a$ and $C_1,\ldots,C_s$ being replaced with all sets $C^x[Z_1,Z_2]$ such that $(Z_1,Z_2)\in \goodx$. Afterwards she proceeds with Sequence II.

\medskip

{\bf Sequence II:} If Walker's graph is already spanning, Walker 
proceeds with Sequence III immediately. Otherwise, Walker plays as follows:
Let $x$ be a vertex with $x\notin V(W)$ for which $\weight(x)$ is largest. Walker includes $x$ to her graph by the following steps:
Let $t\in [2]$ be such that $x\in V_{3-t}$.

\begin{itemize}
	\item[(i)] Walker identifies a vertex $v\in C^x(s_{k,1})$ such that the following properties hold:
	\begin{itemize}
		\item there is a path $P_v$ of length 2 between $a$ and $v$, the edges of which are available, 
		\item there is a $(\cB_t,x)$-structure $\cS_{v,x}$ starting in $v$, the edges of which are available.
	\end{itemize}
	The existence of such a vertex is proven later in the strategy discussion.
	
	\item[(ii)] For her first move, Walker walks along $P_v$ to reach the vertex $v$.
	\item[(iii)] For her next $k$ moves, Walker follows strategy $S_{\text{structure}}$ on the structure $\cS_{v,x}$ until she reaches $x$.
	\item[(iv)] Finally, Walker takes at most $k+1$ further moves to return to vertex $a$ and then proceeds to Sequence III.
\end{itemize}

\medskip

{\bf Sequence III:} Let $b:=8k+14$. If Walker has already tossed a coin on every edge, she stops playing. Otherwise she does the following.
Before doing any move, Walker makes an update to the simulated game MinBox$(n,4pn,0.5\ln^{-1}(n),2b)$:
Let $e_1,\ldots,e_s$ be the edges that Breaker claimed since the last time
when Walker identified an exposure vertex. Then, for each vertex $v\in V(G)$
increase $w_B(J_v)$ by the cardinality of $\{i\leq s:~ v\in e_i\}$. Walker now plays as follows:

\smallskip

\underline{Case 1:} After the update, let there be a free active box in MinBox$(n,4pn,0.5\ln^{-1}(n),2b)$.
Walker first chooses an exposure vertex $x\in V(G)$ for which $J_x$ is a free active
box in MinBox$(n,4pn,0.5\ln^{-1}(n),2b)$ and such that
$$
\dang(J_x):=w_B(J_x)-2b\cdot w_M(J_x)
$$
is largest. She then increases $w_M(J_x)$ by one in the simulated MinBox game (i.e.~in the MinBox game she imagines to claim an element in the box $J_x$), and afterwards plays on $G$ as follows:
Let $t\in [2]$ be such that $x\in V_{3-t}$.
\begin{itemize}
	\item[(i)] Walker identifies a vertex $v\in C^x(s_{k,1})$ such that the following properties hold:
	\begin{itemize}
		\item there is a path $P_v$ of length 2 between $a$ and $v$, the edges of which are available, 
		\item there is a $(\cB_t,x)$-structure $\cS_{v,x}$ starting in $v$, the edges of which are available.
	\end{itemize}
	The existence of such a vertex is proven later in the strategy discussion.
	
	\item[(ii)] For her first move, Walker walks along $P_v$ to reach the vertex $v$.
	\item[(iii)] For her next $k$ moves, Walker follows strategy $S_{\text{structure}}$ on the structure $\cS_{v,x}$ until she reaches $x$.
	
	\item[(iv)] Having that Walker's position is the exposure vertex $x$, Walker starts with the exposure process:
	
	For this, she first fixes a random ordering
	$\pi: [|U_x|] \rightarrow U_x$ of the vertices in $U_x$. According to that ordering,
	she then tosses her coin on the vertices of $U_x$, independently at random, such that for each coin toss the probability of success equals $\ln^{-1}(n)$. Moreover, she tosses her coin until either the first success happens or until she tossed the coin on every element of $U_x$ without any success. According to these different outcomes, Walker considers the following subcases in her strategy.
	
	\begin{itemize}
		\item \textit{No success (failure of type I):} If none of the coin tosses happens to be a success, Walker declares this exposure round as a failure of type I. In particular, she increases the counter $f_I(x)$ by 1,
		and makes an arbitrary move in which she ends in the vertex $x$ again.
		In the auxiliary game MinBox$(n,4pn,0.5\ln^{-1}(n),2b)$ she imagines to receive $2pn\ln^{-1}(n)$ elements of the box $J_x$ (or all of the remaining free elements of $J_x$ if there are less than $2pn\ln^{-1}(n)$), so that $J_x$ is no longer a free active box.
		Moreover, in the Walker-Breaker game on $G$, she sets $U_x=\varnothing$ and removes $x$ from every set $U_w$ with $w\neq x$.
		
		\item \textit{Success without failure:} Assume that for some $k\in\mathbb{N}$, the vertex $\pi(k)$ is the first vertex in $U_x$ for which the coin toss is a success, and assume further that
		$x\pi(k)$ is still a free edge. Then Walker walks along this edge in both directions and hence claims $x\pi(k)$. Afterwards, the following updates happen: for every $i \in [k]$, the vertex $x$ gets removed from $U_{\pi(i)}$ and the vertex $\pi(i)$ gets removed from $U_x$.
		Additionally, in the MinBox game, Maker claims an arbitrary free element of the box $J_{\pi(k)}$, i.e.~$w_M(J_{\pi(k)})$ is increased by one.
		
		\item \textit{Success with failure (failure of type II):} Assume that for some $k\in\mathbb{N}$, the vertex $\pi(k)$ is the first vertex in $U_x$ for which the coin toss is a success, and assume further that $x\pi(k)$ is not a free edge any more.
		Then Walker makes an arbitrary move in which she ends in $x$ again, and she declares this exposure round as a failure of type II. Afterwards, the following updates happen:
		Walker increases both counters $f_{II}(x)$ and $f_{II}(\pi(k))$ by 1. For every $i \in [k]$, the vertex $x$ gets removed from $U_{\pi(i)}$ and the vertex $\pi(i)$ gets removed from $U_x$.
		Additionally, in the MinBox game, Maker claims an arbitrary free element of the box $J_{\pi(k)}$, i.e.~$w_M(J_{\pi(k)})$ is increased by one.

	\end{itemize}

	\item[(v)] Finally, Walker takes at most $k+1$ further moves to return to vertex $a$ and then proceeds to Sequence I.
\end{itemize}

\underline{Case 2:} After the update, let there be no free active box in MinBox$(n,4pn,0.5\ln^{-1}(n),2b)$. Then Walker tosses her coin on every edge $uv \in E(G)$ on which she did not toss a coin yet. If the coin tossed for $uv$ is successful, she increases
$f_{II}(u)$ and $f_{II}(v)$ by one. Afterwards, Walker stops playing the game.

\subsection{Strategy discussion}\label{subsec:discussion}

The goal of this subsection is to prove that Walker can always follow the described strategy and that by following this strategy, Walker a.a.s.~ensures that~\eqref{eq:resilience_condition} holds by the end of the game. As explained in Subsection~\ref{subsec:setup}, this is enough to conclude Theorem~\ref{thm:main.general}.

We start by proving some statements that are guaranteed to hold as long as Walker can follow her strategy.

\begin{clm}[Maintaining properties I]\label{clm:maintain1}
	As long as Walker can follow the strategy described in Subsection~\ref{subsec:strategy}, the following holds:
	Breaker claims at most $8k+14$ edges between any two consecutive sequences of the same type (I, II or III).
\end{clm}

\begin{proof}
	Each application of Sequence I, II, and III lasts at most $2, 2k+2$, and $2k+3$ rounds for Walker, respectively. Therefore, the number of edges claimed by Breaker between two consecutive sequences of the same type are at most $2(2+2k+2+2k+3)=8k+14$.
\end{proof}

\begin{clm}[Maintaining properties II]\label{clm:maintain2}
	As long as Walker can follow the strategy described in Subsection~\ref{subsec:strategy}, the following holds:
	For every $x\in V(G)\setminus V(W)$, it holds that $\weight(x)<\ln^2(n)$. In particular, taking $t\in [2]$ such that $x\in V_{3-t}$, Breaker claims at most
	$\ln^4(n)$ edges below level $k-1$ that see $x$ with respect to $\cB_t$ and claims at most $n^{1/3+\eps/2}$ edges between level $k-1$ and $k$ that see $x$ with respect to $\cB_t$.
\end{clm}

\begin{proof}
	While the Walker-Breaker game is proceeding, consider an auxiliary Continuous Box Game with a box $F_x$ for each vertex 
	$x\in V(G)\setminus \{a\}$. Let CMaker and CBreaker play as follows:
	\begin{enumerate}
		\item[(1)] If Breaker claims an edge $e$ in the Walker-Breaker game on $G$, then do the following: for every $t\in [2]$ and $x\in V_{3-t}$ such that $e$ sees $x$ w.r.t.~$\cB_t$, CMaker
		adds $\weight_t(e)$ to the box $F_x$.
		\item[(2)] If Walker chooses a vertex $x$ according to Sequence II, i.e.~such that $\weight(x)$ is largest, then do the following: at the beginning of the sequence, i.e.~already before Walker starts to play according to the steps (i)--(iv), CBreaker destroys box $F_x$ in the 
		auxiliary Continuous Box Game.
	\end{enumerate}
	
	Then we observe the following. With every edge $e$ that Breaker claims, CMaker adds a total weight of at most $2$ over all the boxes. This follows from the definition of the weights and the property (E) in Lemma~\ref{lemma:technical}. Indeed, let $e$ be any edge that Breaker claims, and let $t\in [2]$. 
	If $e$ is an edge below level $k-1$ w.r.t.~$\mathcal{B}_t$,
	then $\weight_t(e)=\ln^{-2}(n)$ by definition, and by (E1) from Lemma~\ref{lemma:technical} we know that $e$ sees at most $\ln^2(n)$ vertices $x\in V_{3-t}$ w.r.t.~$\cB_t$.
	Similarly, if $e$ is an edge between levels $k-1$ and $k$ w.r.t.~$\mathcal{B}_t$, then $\weight_t(e)=n^{-1/3-0.1\eps}$ by definition, and by (E2) from Lemma~\ref{lemma:technical} we know that $e$ sees at most $n^{1/3+0.1\eps}$ vertices $x\in V_{3-t}$ w.r.t.~$\cB_t$.
	Hence, in any case, a total weight of at most $1$ is added over all boxes $F_x$ with $x\in V_{3-t}$, and hence at most $2$ over all boxes.
	
	Now, since between any two applications of Sequence II
	Breaker claims at most $8k+14$ edges, we obtain that
	CMaker adds a total weight of at most $2\cdot (8k+14)=16k+28$
	over all boxes, before CBreaker destroys the next box.
	Hence, we can consider the bias to be $(16k+28:1)$ for this auxiliary game. 
	Moreover, by the description in (2) we also know that
	CBreaker follows the strategy $\mathcal{S}$ described shortly before Lemma~\ref{lem:Cboxgame}. Therefore, Lemma~\ref{lem:Cboxgame} guarantees that CMaker never achieve to have more than $(16k+28)(\ln(n)+1)$ weight in any box
	which is still not destroyed. We conclude that we have
	$$\weight(x)\leq (16k+28)(\ln(n)+1) + (16k+28) < \ln^2(n)$$
	as long as $x\notin V(W)$. Note that the additional factor of $(16k+28)$ is added to also have a bound for all the intermediate rounds in the Walker-Breaker game in which we do not make an update in the Continuous Box Game. 
	
	Now, fix any $x\notin V(W)$ and let $t\in [2]$ be such that $x\in V_{3-t}$. Since every edge below level $k-1$ that sees $x$ with respect to $\cB_t$ has $\weight_t(e)=\ln^{-2}(n)$,
	it follows from $\weight(x)<\ln^2(n)$ that Breaker cannot have more than $\ln^4(n)$ of these edges.
	Similarly, since every edge between level $k-1$ and $k$ that sees $x$ with respect to $\cB_t$ has $\weight_t(e)=n^{-1/3-0.1\eps}$,
	it follows that Breaker cannot claim more than $n^{1/3+0.1\eps}\ln^2(n) < n^{1/3+\eps/2}$ of these edges. 
\end{proof}

Having the above claim in hands, we are now able to prove that Walker can always follow the proposed strategy. We start with Sequence I.

\begin{clm}[Following Sequence I]
	Walker can always follow the proposed strategy of Sequence I.
\end{clm}

\begin{proof}
	If Case 1 happens, i.e.~$|N_a|<n^{1/3}$,
	then this means that Sequence I happened less than $n^{1/3}$ times so far, and by Claim~\ref{clm:maintain1} Breaker has at most
	$n^{1/3}(8k+14)$ edges in total. Moreover, by property (S) of Lemma~\ref{lemma:technical}, we know that $|N(a,R)|\geq n^{1/3+\eps/2}$.
	Hence, Walker can find a vertex $w$ as described in step (i) of this case and hence follows the strategy. 
	
	If otherwise Case 2 happens, then it remains to check that Walker can indeed follow the strategy $\mathcal{S}_{paths}$ as proposed, i.e.~check the requirements of Lemma~\ref{lemma:paths_beck}. That is, we need to check 
	that $|\goodx|\leq \exp(n^{1/3+\eps})$ and that, when Case 2 happens for the first time, for each $(Z_1,Z_2)\in \goodx$ we have at least $n^{1/3+1.1\eps}$ available edges between
	$C^x[Z_1,Z_2]$ and $N_a$. For the first requirement observe that the number of pairs $(Z_1,Z_2)\in \goodx$ is at most
	$$
	\binom{n^2}{\ln^4(n)}\cdot \binom{n^2}{n^{1/3+\eps/2}} < \exp(n^{1/3+\eps})
	$$
	for large $n$. For the number of available edges, observe that
	at most $n^{1/3}(8k+14)$ edges were claimed by Breaker when Case 2 happens for the first time, while Property (R) ensures that
	$e_G(N_a,C^x[Z_1,Z_2])\geq n^{1/3+1.5\eps}$ for every $(Z_1,Z_2)\in \goodx$. Hence, more than
	$n^{1/3+1.1\eps}$ of these edges are still free.
\end{proof}

By following Sequence I, Walker additionally maintains the following property.

\begin{clm}[Maintaining properties III]\label{clm:maintain3}
	From the moment when $|N_a|=n^{1/3}$ holds for the first time,
	the following always happens:
	Let $t\in [2]$ and $x\in V_{3-t}$. Let 
	$(Z_1,Z_2)\in \goodx$, then there is a path $(a,y,v)$ of length 2 such that both its edges are available and $y\in N_a$ and $v\in C^x[Z_1,Z_2]$. 
\end{clm}

\begin{proof}
	Since Walker can always follow the strategy of Sequence I, and hence of $\mathcal{S}_{paths}$ from the moment when $|N_a|=n^{1/3}$, we know that by the end of the game she may have an edge between $N_a$ and each of the mentioned sets $C^x[Z_1,Z_2]$.
	In particular, this means that she eventually claims a path $(a,y,z)$ of length 2 such that $y\in N_a$ and $z\in C^x[Z_1,Z_2]$. 
	Hence, the claim follows as all free edges and all Walker's edges are available.
\end{proof}

Let us prove next that Walker can always follow the strategy
when she needs to move according to Sequence II or Sequence III.

\begin{clm}[Following Sequences II and III]
	Walker can always follow the proposed strategy of Sequences~II and III.
\end{clm}

\begin{proof}
	Let us consider Sequence II first and assume that Walker so far could always follow her strategy. Recall that the sequence starts in the vertex $a$. Moreover, assume further that Walker has fixed a vertex $x$ according to the description of Sequence II, and let $t\in [2]$ such that $x\in V_{3-t}$. 
	
	We start by checking that Walker can follow step (i). For this,
	before the sequence starts, denote with $Z_1$ and $Z_2$ the edges of Breaker which are below level $k-1$ or between level $k-1$ and $k$, respectively, which see vertex $x$ with respect to $\cB_t$. By Claim~\ref{clm:maintain2}, we have $(Z_1,Z_2)\in \goodx$. We may consider two cases.
	First, we may assume that $|N_a|<n^{1/3}$. Then Sequence I happened less than $n^{1/3}$ times so far and, following Claim~\ref{clm:maintain1}, at most $(8k+14)n^{1/3}$ rounds have been played. In particular, we can find a set $A\subseteq N_G(a,R)$
	of size $n^{1/3}$ such that all edges of $E_G(a,A)$ are free.
	Moreover, by Property~(R) of Lemma~\ref{lemma:technical} it follows that there are at least $n^{1/3+1.5\eps}$ edges between $A$ and $C^x[Z_1,Z_2]$ in $G$. Among these edges, there must be free edges since at most $(8k+14)n^{1/3}$ rounds have been played so far; denote one such edge with $yv$ such that $y\in A$ and $v\in C^x[Z_1,Z_2]$. Then the path $P_v=(a,y,v)$ consists of available edges. 
	Moreover, since $v\in C^x[Z_1,Z_2]$ and because of Observation~\ref{obs:CgivesSk}, there exists a $(\cB_t,x)$-structure $\mathcal{S}_{v,x}$ incident with $v$, the edges of which are not claimed by Breaker. 
	Assume then that $|N_a|=n^{1/3}$. Then, by Claim~\ref{clm:maintain3}, we find a path $P_v=(a,y,v)$ consisting of available edges such that $y\in N_a$ and $v\in C^x[Z_1,Z_2]$. Hence, Walker can follow (i) as before.
	
	Now, step (ii) can easily be followed, since the edges of $P_v$ are available when Walker starts the sequence. Step (iii) can be done by Lemma~\ref{lemma:structure_strategy}. Note that steps (i)--(iii) only need $k+1$ moves by Walker. By returning on the same edges to $a$, Walker can easily follow step (iv). 
	
	\smallskip
	
	Hence, let us turn to Sequence III. Its discussion is almost the same as the discussion of Sequence II, except that step (iv) in this sequence is new. For this step note that whenever a failure happens, i.e.~no success on an edge or success on an edge which is not free, then the strategy allows Walker to walk along an arbitrary available edge incident with $x$. (At least one such edge exists since Walker just reached $x$ and hence occupies an edge incident with $x$.) Otherwise, if no failure happens, i.e.~success on a free edge, then the strategy asks Walker to claim this free edge which is incident with her current position $x$. Hence, in any case, Walker can follow that part of the strategy.
\end{proof}

Now, since Walker can always follow the strategy of Sequences I~and II, it follows that by the end of the game her graph is spanning. In order to show that it is very likely to achieve~\eqref{eq:resilience_condition}, we now take a closer look at the randomized moves of step (iv) in Sequence III. The following four claims are analogues of Claims~3.1--3.4
in~\cite{F15Gen}.

\begin{clm}[Analysing random moves I] \label{clm:ranmovesI}
	As long as Walker does not play according to Case 2 of Sequence III, for every $v \in V(G)$ it holds that $w_B(J_v)<2pn$ and $w_M(J_v)<2pn(1 + \ln^{-1}(n))$.
\end{clm}

\begin{proof}
	The bound on $w_B(J_v)$ follows from Claim~\ref{clm:degrees} and the fact that Breaker claims an element of $J_v$ in the MinBox game if and only if he claims an edge incident to $v$. We observe that $w_M(J_v)$ is updated under three scenarios. It is increased by one whenever $v$ is the exposure vertex. It is also increased by one if $v$ is not the exposure vertex but the coin shows success on edge incident with $v$. Since the maximum degree in $G$ is smaller than $2pn$ (Claim~\ref{clm:degrees}), these first two scenarios contribute at most $2pn$ to $w_M(J_v)$. Moreover, $w_M(J_v)$ is increased by up to $2pn\ln^{-1}n$ whenever Walker encounters a failure of type I at vertex $v$; however this may happen at most once. Hence, $w_M(J_v)$ cannot become larger than $2pn(1+\ln^{-1}(n))$.
\end{proof}

\begin{clm}[Analysing random moves II]\label{clm:ranmovesII}
	For every vertex $v \in V(G)$, the box $J_v$ becomes inactive before 
	$d_B(v) \geq \eps pn/5$.
\end{clm}

\begin{proof}
	Assume to the contrary that at some point in the game there is an active box $J_v$ with $w_B(J_v)=d_B(v) \geq \eps pn/5$. In the first round in which the size of $\eps pn/5$ gets exceeded we must have
	$\eps pn/5\leq w_B(J_v)\leq \eps pn/5 + b$.
	At that moment $J_v$ must still be free, since $w_M(J_v)<2pn(1 + \ln^{-1}(n))$ by the previous claim and since $|J_v|=4pn$ by definition.	
	But then, by Theorem~\ref{theorem:MinBox}, $\dang(J_v)=w_B(J_v)-2b\cdot w_M(J_v) \leq 2b (\ln(n)+1)$ which rearranges to give $w_M(J_v) \geq \frac{w_B(J_v)}{2b} - (\ln (n)+1) \geq \frac{\eps pn }{10b} - (\ln (n)+1) \geq 0.5 \ln^{-1}(n) \cdot 4pn$, for large enough $n$. However, this contradicts the assumption that $J_v$ is still active.
\end{proof}

\begin{clm}[Analysing random moves III]\label{clm:ranmovesIII}
	A.a.s.~the following happens: For every vertex $v \in V(G)$, as long as 
	$U_v \neq \varnothing$, the box $J_v$ is active. Hence, a.a.s.~no edge of $G$ will get exposed because of Case 2 in Sequence III.
\end{clm}

\begin{proof}
	Applying Claim~\ref{clm:degrees} for $H\sim G_{\ln^{-1}(n)}$, we have that a.a.s.~every vertex $v\in V(H)$ satisfies $d_H(v)\leq (1\pm 2\eps)pn \ln^{-1}(n)$. In particular, the following is a.a.s.~true: fixing any vertex $v$, there happen to be at most $(1\pm 2\eps)pn \ln^{-1}(n)$ incident edges on which the coin tossing is a success. From now on, condition on this property.
	
	Assume that the statement of the claim is wrong, i.e.~at some moment in the game there is a vertex $v$ for which we have $U_v \neq \varnothing$ and the box $J_v$ is inactive. 
	The assumption $U_v \neq \varnothing$ means that not every edge at $v$ has been exposed yet, and hence no failure of type I has happened where $v$ was the exposure vertex. The assumption of $J_v$ being inactive leads to $w_M(J_v)\geq 2pn \ln^{-1}(n)$ by definition. But then, since $f_I(v)=0$, there must have been $2pn \ln^{-1}(n)$ edges incident with $v$ on which the coin tossing was successful, a contradiction. 
\end{proof}

\begin{clm}[Analysing random moves IV]\label{clm:ranmovesIV}
	A.a.s.~the following happens: For every vertex $v \in V(G)$, we maintain
	$f_{II}(v)\leq 0.5 \eps pn\ln^{-1}(n)$. 
\end{clm}

\begin{proof}
	
	By Claim~\ref{clm:ranmovesIII} we have that a.a.s.~all edges of $H\sim G_{\ln^{-1}(n)}$ get exposed because of Case~1 in Sequence III. Moreover, by Claim~\ref{clm:ranmovesII}, for every $v\in V(H)$ we have
	$d_B(v)\leq \eps pn/5$ as long as the corresponding box $J_v$ is active.
	Since $f_{II}(v)$ counts the number of failures of type II at $v$, i.e.~the number of successes on edges incident with $v$ which have already been claimed by Breaker, we observe that $f_{II}(v)$ is stochastically dominated by $\operatorname{Bin}(\eps pn/5,\ln^{-1}(n))$. Applying a Chernoff inequality (Lemma~\ref{lem:Chernoff1}), we conclude that
	$$
	\mathbb{P}(f_{II}(v) \geq 0.5 \eps pn\cdot \ln^{-1}(n))
	\leq 
	e^{-\eps pn/15\cdot \ln^{-1}(n)}
	\leq
	e^{- n^{1/3}}
	$$
	for large enough $n$. The claim now follows by taking a union bound over all vertices $v$.
\end{proof}

Finally, let us explain why the above statement implies that a.a.s.~\eqref{eq:resilience_condition} holds by the end of the game.
According to Claim~\ref{clm:degrees} (applied for the random graph $H\sim G_{n,p\ln^{-1}(n)}$) we have $d_H(v)\geq (1-\eps)pn\ln^{-1}(n)$ for every $v\in V(G)$. From these at least $(1-\eps)pn\ln^{-1}(n)$ edges a.a.s.~at most $0.5 \eps pn \ln^{-1}(n)$ are failures of type II and therefore Breaker's edges, while the remaining end up in Walker's graph. This gives
$$
d_{H\cap W}(v) \geq d_H(v) - 0.5 \eps pn \ln^{-1}(n) \geq (1-\eps)d_H(v)
$$
for every $v\in V(H)$. Thus, Theorem~\ref{thm:main.general} is proven.\hfill $\Box$

\section{Proof of Main Technical Lemma}\label{sec:technical}

Before we prove Lemma~\ref{lemma:technical}, let us first state and prove two simple claims that we will require later. As before, we do not intend to optimize polylogarithmic factors. Hence, in order to simplify calculations, we may sometimes make generous estimates.

\begin{clm}\label{claim:neighbours}
	Let $\eps>0$, let $n$ be a large enough integer, and let $p=n^{-2/3+\eps}$. Moreover, let $A,B\subset [n]$ be any disjoint sets such that $1\leq |A|\leq p^{-1}$ and $|A||B|> n^{2/3}$. When we reveal the edges of
	$G\sim G_{n,p}$ on vertex set $[n]$, then with probability 
	at least $1-\exp(-0.5\ln^2(n))$ it holds that 
	\begin{equation*}
		|N_G(A)\cap B|=\ln^{\pm 2}(n) p|A||B|\, .
	\end{equation*}
\end{clm}

\begin{proof}
	Let $\delta\in (0,0.1)$, and whenever necessary assume $n$ to be large enough.
	
	At first, we observe that $e_G(A,B)\sim Bin(|A||B|,p)$ with expectation $\Exp(e_G(A,B))=|A||B|p$. Hence, using Chernoff (Lemma~\ref{lem:Chernoff1}) 
	and $|A||B|p>n^{\eps}$, we get 
	with probability at least $1-\exp(-n^{\eps/2})$ that
	\begin{equation}\label{eqn:5.1}
		e_G(A,B)=(1\pm \delta) |A||B|p\, .
	\end{equation}

	Furthermore, for every
	$v\in B$ we have $d_G(v,A)\sim Bin(|A|,p)$ and
	$\Exp(d_G(v,A))=p|A|\leq 1$. Hence, applying Chernoff (Lemma~\ref{lem:Chernoff2}) and union bound, we see that
	with probability at least $1-n\exp(-0.9\ln^2(n))$, 
	\begin{equation}\label{eqn:5.2}
		d_G(v,A)\leq 0.9\ln^2(n) ~~ \text{for every }v\in B\, .
	\end{equation}

	The probability that at least one of the above events fails
	can be bounded with $\exp(-0.5\ln^2(n))$ by the union bound. Hence, we may assume that both equations~\ref{eqn:5.1} and \ref{eqn:5.2} hold. Then, using that
	$$
	\frac{e_G(A,B)}{\max_{v\in B} d_G(v,A)} \leq |N_G(A)\cap B| \leq e_G(A,B)\, ,
	$$
	the claim follows.
\end{proof}

\begin{clm}\label{claim:nextlevel}
	Let $\delta,\eps>0$, $n$ be a large enough integer, and $p=n^{-2/3+\eps}$. Moreover, let $M_1$, $M_2$, $M_3$, $B_1^{\ast}$, $B_2^{\ast}$, $B_3^{\ast}$, $B\subset [n]$ be any disjoint sets such that $1\leq |M_j|\leq n^{1/3-3\eps}$ and $|B|,|B_j^\ast|\geq \delta n$ for every $j\in [3]$. When we reveal the edges of
	$G\sim G_{n,p}$ between these seven sets, then with probability 
	at least $1-\exp(-0.4\ln^2(n))$ the following holds.
	
	\begin{enumerate}
		\item[(P)] Let $C\subseteq B$ be the set of vertices $b\in B$ such that for every $j\in [3]$ there exists a path $(b,y_j,z_j)$ with $y_j\in B_j^\ast$ and $z_j\in M_j$. Then
		$
		|C| = \ln^{\pm 13}(n) n^{6\eps} \prod_{j\in [3]} |M_j|\, .
		$
	\end{enumerate}
\end{clm}

\begin{proof}
	Whenever necessary assume $n$ to be large enough.
	
	We may write $A_j:=N_G(M_j)\cap B_j^\ast$. Then $C$ is the set of vertices $b\in B$ that have a neighbour in each of the sets $A_j$ with $j\in [3]$. In order to obtain estimates on $|C|$ we may apply Claim~\ref{claim:neighbours} repeatedly. For this, we first reveal the edges of $G\sim G_{n,p}$ between $M_j$ and $B_j^\ast$ for each $j\in [3]$. Then Claim~\ref{claim:neighbours} (applied with sets $M_j$ and $B_j^\ast$) yields that it is likely to have
	\begin{equation}\label{eq:nextlevel:step0}
		|A_j|=\ln^{\pm 2}(n) p|M_j||B_j^\ast|\, 
	\end{equation}
	for every  $j\in [3]$. We may condition on these events and note that then 
	\begin{equation*}
		n^{1/3}<|A_j|< p^{-1}
	\end{equation*}
	holds. Next, we proceed as follows: in a first step we reveal the edges between  $A_1$ and $B$ to find $B_1:=N_G(A_1)\cap B$;
	in a second step we reveal the edges between
	$A_2$ and $B_1$ to find $B_2:=N_G(A_2)\cap B_1$;
	and in a third step we reveal the edges between
	$A_3$ and $B_2$ to find $C=N_G(A_3)\cap B_2$.
	Each time we can apply Claim~\ref{claim:neighbours},
	provided that the two relevant sets satisfy the size conditions
	given by Claim~\ref{claim:neighbours}. 
	
	In the first step, applying Claim~\ref{claim:neighbours} with sets $A_1$ and $B$, we then obtain that is likely to have
	\begin{equation}\label{eq:nextlevel:step1}
		|B_1|=\ln^{\pm 2}(n) p|A_1||B|
	\end{equation}
	which, using $|A_1|>n^{1/3}$, gives $|B_1|>n^{2/3}$.
	If we condition on this, then $A_2$ and $B_1$ satisfy the size conditions of Claim~\ref{claim:neighbours}, and hence,
	in the second step we obtain that is likely to have
	\begin{equation}\label{eq:nextlevel:step2}
		|B_2|=\ln^{\pm 2}(n) p|A_2||B_1|\, .
	\end{equation}
	In particular, using $|A_2|,|A_3|>n^{1/3}$, we then have 
	$|A_3||B_2|>n^{2/3}$. So, if we condition on this,
	then $A_3$ and $B_2$ satisfy the size conditions of Claim~\ref{claim:neighbours}, and for the third step it is likely 
	to have
	\begin{equation}\label{eq:nextlevel:step3}
		|C|=\ln^{\pm 2}(n) p|A_3||B_2|\, .
	\end{equation}
	
	By Claim~\ref{claim:neighbours} and the union bound, 
	with probability at least $1-\exp(-0.4\ln^2(n))$,
	all of the events described by \eqref{eq:nextlevel:step0}-\eqref{eq:nextlevel:step3} hold at the same time. 
	Putting everything together we get
	\begin{align} 
		|C| & \stackrel{\eqref{eq:nextlevel:step3}-\eqref{eq:nextlevel:step1}}{=} \ln^{\pm 6}(n) p^3 |B| \prod_{j\in [3]} |A_j| 
		\stackrel{\eqref{eq:nextlevel:step0}}{=} \ln^{\pm 12}(n) p^6 |B| \prod_{j\in [3]} |M_j||B_j^\ast|  \label{eq:nextlevel:step4} \\
		& = \ln^{\pm 13}(n) p^6 n^4 \prod_{j\in [3]} |M_j| 
		= \ln^{\pm 13}(n) n^{6\eps} \prod_{j\in [3]} |M_j| \nonumber
	\end{align}
	which is as claimed.
\end{proof}

\subsection{Main proof} We now prove Lemma~\ref{lemma:technical}.
Let $\eps\in\mathcal{R}$ be  given and set $k:=\log_3(2\eps^{-1}+12)-2$. Let $V=[n]$ be the vertex set, and before revealing any edges of $G_{n,p}$, do the following: fix an arbitrary vertex $a\in [n]$ and an equipartition $V\setminus \{a\}=V_1\cup V_2$, as well as
families $\mathcal{B}_t=\left\{B_t(s):~ s\in V(\cS_k-s_0)\right\}$ of pairwise disjoint vertex subsets of $V_t$ each of size $\frac{n}{3^{k+10}}$. Furthermore set 
$V(\mathcal{B}_t)=\bigcup_{B\in \mathcal{B}_t} B$ for $t\in [2]$,
and set
$$R = V\setminus (V(\mathcal{B}_1)\cup V(\mathcal{B}_2) \cup \{a\})\, .$$

We now start revealing the edges of $G\sim G_{n,p}$ and
prove that a.a.s.~all properties listed in Lemma~\ref{lemma:technical} hold.

\underline{\bf Property (S):} By assumption we have
$|B_t(s)|=\frac{n}{3^{k+10}}$ for every $s\in V(\cS_k-s_0)$ 
and $t\in [2]$. Since $v(\mathcal{S}_k) = 2\cdot 3^{k}- 1$ by Claim~\ref{claim:structure_size}, it follows that $|R|\geq \frac{n}{2}$. Moreover, by revealing only the edges between $a$ and $R$, we have $|N(a,R)|\sim$Bin$(|R|,p)$ with expectation at least $\frac{1}{2}n^{1/3+\eps}$.
Hence, by Chernoff (Lemma~\ref{lem:Chernoff1}), Property (S) fails with probability at most $\exp(-0.1 n^{1/3+\eps})$.

\underline{\bf Property (C):} Let $t\in [2]$ and $x\in V_{3-t}$. We prove that for this particular choice of $t$ and $x$, the described property in (C) fails with probability at most $\exp(-0.3 \ln^2(n))$. Thereafter, the union bound concludes the argument.

Note that for this property we only need to reveal the edges in $V(\mathcal{B}_t)$ and the edges between $x$ and boxes $B_t(s)$ with $s\in L^\ast_0$, and hence checking this property is independent of (S). 
Set $C^x(s_0)=\{x\}$. Starting from this candidate set,
we find the candidate sets $C^x(s) \subseteq B_t(s)$ with $s\in L_{\ell}$ and $\ell\in [k]$ iteratively, starting with $L_1$, then $L_2$ and so on, always only revealing the edges which are needed.

To be more precise, one step in the iteration looks as follows: Let $\ell\in [k]$ and assume we have already fixed the candidate sets $C^x(s)$ with $s\in L_{\ell -1}$. Then we reveal all edges between boxes $B_t(s)$ with $s\in L_{\ell -1}\cup L_{\ell -1}^\ast \cup L_{\ell}$. For every $i\in [3^{k-\ell}]$ we then let $C^x(s_{\ell,i})$ be the set of all vertices $v \in B_t(s_{\ell,i})$ such that for every $j \in \{0,1,2\}$ it holds that there exists a path $(v,y_j,z_j)$ with $y_j \in B_t(s^*_{\ell-1,3i-j})$ and $z_j \in C^x(s_{\ell-1,3i-j})$. Afterwards we let $C^x(s^*_{\ell-1,3i-j})$ be the union over of all $v_j$ which appear in such a path.

Properties $(C2)$ and $(C3)$ immediately hold by definition. To show 
that $(C1)$ is likely to hold, we apply Claim~\ref{claim:nextlevel} (with $\delta = 1/3^{k+10}$) along the iteration. That is, in the $\ell$-th step of the above iteration, having $i\in [3^{k-\ell}]$, we set $M_j = C^x(s_{\ell-1,3i-j})$, $B^*_j = B_t(s^*_{\ell-1,3i-j})$, and $B = B_t(s_{\ell,i})$. By Claim~\ref{claim:nextlevel} it is then likely to have
\begin{equation}\label{eq:candidate:recursion}
	|C^x(s_{\ell,i})|=\ln^{\pm 13}(n) n^{6\eps} \prod_{j\in [3]} |C^x(s_{\ell-1,3i-j})|
\end{equation}
from which we may conclude inductively that 
\begin{equation}\label{eq:candidate:explicit}
	|C^x(s_{\ell,i})|=n^{(3^{\ell+1}-3)\eps} \ln^{\pm 3^{3\ell}}(n)\, .
\end{equation}
Note that in this whole iteration the above application of Claim~\ref{claim:nextlevel} happens exactly once for each vertex $s\in V(\mathcal{S}_k-s_0)$. Thus, the probability that \eqref{eq:candidate:recursion} and hence \eqref{eq:candidate:explicit} fail for some $(\ell,i)\in I(k)$ is bounded by $v(\mathcal{S}_k)\cdot \exp(-0.4 \ln^2(n)) < \exp(-0.3 \ln^2(n))$, provided that each time when we apply Claim~\ref{claim:nextlevel} the required condition on $|M_j|$ holds. 
But now, using \eqref{eq:candidate:explicit} along the way, this condition is guaranteed since for every $\ell\in [k]$ and $s\in L_{\ell-1}$ the likely size of $C^x(s)$ is bounded by
$$
n^{(3^{\ell}-3)\eps} \ln^{3^{3\ell-3}}(n)\leq
n^{(3^{k}-3)\eps} \ln^{3^{3\ell-3}}(n)
<
n^{2/9}
$$
by the choice of $k$, and provided that $n$ is large enough.

\underline{\bf Property (E):} Let $t\in [2]$ and $\ell\in [k]$, and fix any edge $e$ that appears below level $\ell$ with respect to $\mathcal{B}_t$. For this fixed edge, we may prove that (E1) or (E2) fail with probability at most  $\exp(-0.5\ln^2 n)$.
Hence, applying union bound over all edges, (E) fails 
with probability at most  $\exp(-0.4\ln^2 n)$.

By Definition~\ref{dfn:structures.appearance}, the edge $e$ can only see a vertex $x\in V_{3-t}$ if contained in a $(\mathcal{B}_t,x)$-structure. Hence, we may assume that there exist adjacent vertices $s,s'\in V(\mathcal{S}_k-s_0)$ such that $e$ intersects both $B_t(s)$ and $B_t(s')$, and such that $s\in L^\ast_{\ell'}$ for some $\ell'\leq \ell$. If $s'=x$, then $x$ would be the only vertex in $V_{3-t}$ that $e$ could see, and hence the bounds in (E1) and (E2) would hold trivially. Therefore, we may assume further that $e\cap V_{3-t}=\varnothing$. 
Now, for such $e$ we have that if $e$ sees a vertex $x \in V_{3-t}$, then there must be subgraph $S\subseteq \mathcal{S}_k$ isomorphic to $\mathcal{S}_\ell - s_0$ s.t. the following properties hold:
\begin{itemize}
	\item[(a)] $e \in E(S)$ and the vertices of $e$ are the copies of $s$ and $s'$,
	\item[(b)] $V(S) \subseteq V(\mathcal{B}_t)$, and
	\item[(c)] each leaf of $S$ is a neighbour of $x$.
\end{itemize}

Note that the properties (a) and (b) depend on edges in $E(V(\mathcal{B}_t))$ only, while (c) depends on edges in $E(V(\mathcal{B}_t),V_{3-t})$. So, we first can expose all edges in $E(V(\mathcal{B}_t))$ and give an upper bound on the number of possible structures $S$ satisfying both (a) and (b). Afterwards we may expose the edges in $E(V(\mathcal{B}_t),V_{3-t})$ and bound the number of vertices $x$ such that (c) is satisfied.

For the first part in turns out that it suffices to have a reasonable upper bound on the degrees in $V(\mathcal{B}_t)$. If we reveal the edges in $E(V(\mathcal{B}_t))$, then by a simple Chernoff argument we have that with probability at least $1 - \exp(-n^{1/3})$ every vertex 
$v \in V(\mathcal{B}_t)$ satisfies $d(v,V(\mathcal{B}_t)) \leq np$.
If we condition on this bound and use $v(S_\ell - s_0) = 2 \cdot 3^{\ell} - 2$ from Claim~\ref{claim:structure_size}, it follows that
there are at most 
$$(np)^{v(S_\ell - s_0) - 2} = (np)^{2 \cdot 3^\ell - 4}$$ 
copies $S$ of $S_\ell - s_0$ fulfilling properties (a) and (b).
Now, condition on this event to hold and expose the edges between $V(\mathcal{B}_t)$ and $V_{3-t}$. Then for any copy $S$ with properties (a) and (b) and any vertex $x \in V_{3-t}$ it holds that 
$$\mathbb{P}(\text{all leaves of }S\text{ are neighbours of }x) =
p^{|L_0^\ast|} = p^{3^\ell}.$$ 
Taking a union bound over all relevant copies of $S$, we then conclude that
$$\mathbb{P} \big( \exists S \text{ with properties (a),(b): all leaves of }S\text{ are in }N(x)\big) \leq (np)^{2\cdot 3^\ell - 4} \cdot p^{3^\ell} =: p^\ast\, .$$
Since these events are independent for all $x \in V_{3-t}$, it follows that the random variable
$$
X_e := |\{ x \in V_{3-t}: \exists S \text{ with properties (a),(b) such that all leaves of }S\text{ are in }N(x) \}|
$$
is stochastically dominated by the random variable Bin$(|V_{3-t}|,p^\ast) =: Y_e$. The expectation of $Y_e$ is
$$
\mathbb{E}[Y_e] = |V_{3-t}| p^* \leq n \cdot (np)^{2\cdot 3^\ell - 4} \cdot p^{3^\ell} = n^{\alpha(\ell)}
$$
with $\alpha(\ell) = 1 + (2 \cdot 3^\ell - 4) + (-\frac23 + \eps)\cdot(3 \cdot 3^\ell - 4) = - \frac{1}{3} + \eps \cdot (3^{\ell + 1} - 4)$.
In light of (E1) and (E2), we now distinguish two cases. 

{\bf Case (E1): $e$ is below level $k-1$ ($\ell \leq k-1$).}
In this case we have 
$\alpha(\ell) \leq  -\frac13 + \eps \cdot (3^{k} - 4) \leq -\frac19$
by our choice of $k$. In particular, we have $\mathbb{E}[Y_e] = o(1)$ and using Chernoff (Lemma~\ref{lem:Chernoff2}) we get
$$
\mathbb{P}(Y_e \geq \ln^2(n)) \leq e^{-\ln^2(n)} 
~~ \Rightarrow ~~
\mathbb{P}(X_e \geq \ln^2(n)) \leq e^{-\ln^2(n)}\, .
$$
Hence, $e$ fails (E1) with probability at most $\exp(-n^{1/3})+\exp(-\ln^2 n) < \exp(-0.5\ln^2 n)$.

{\bf Case (E2): $e$ is between level $k-1$ and $k$ ($\ell = k-1$).}
In this case we have 
$\alpha(\ell) = -\frac13 + \eps \cdot (3^{k+1} - 4) = \frac13$
by our choice of $k$.
In particular, we have $\mathbb{E}[Y_e] = n^{1/3}$ and using Chernoff (Lemma~\ref{lem:Chernoff1}) we get
$$
\mathbb{P}(Y_e \geq n^{1/3 + 0.1 \eps}) \leq e^{-n^{1/3}}
~~ \Rightarrow ~~
\mathbb{P}(X_e \geq n^{1/3 + 0.1 \eps}) \leq e^{-n^{1/3}}\, .
$$
Hence, $e$ fails (E2) with probability at most $\exp(-n^{1/3})+\exp(-n^{1/3}) = 2\exp(-n^{1/3})$.

\smallskip

\underline{\bf Property (R):} 
In order to verify property (R), we first need a bound on the number of candidates in $B_t(s_{k,1})$ for a given vertex $x\in V_{3-t}$, for $t\in[2]$, that belong to a $(\mathcal{B}_t,x)$-structure that
also contains a fixed vertex $v$. For this, we define
$$
c_{\ell} :=
\begin{cases}
	n^{(3^{k+1}-4)\eps}~ , & ~ \text{if }\ell\neq k-1\\
	n^{1/9+4\eps} ~ , & ~ \text{if }\ell =   k-1\, .
\end{cases}
$$
Then the following claim holds.

\begin{clm}\label{clm:technical:contain.v}
	With probability at least $1-\exp(-0.1\ln^2(n))$
	the following property holds:
	For every $t\in [2]$, $x\in V_{3-t}$ and every $v\in B_t(s^\ast_{\ell,i})$, with $\ell\in [k-1]$ and $i\in [3^{k-i}]$,
	there are at most $c_{\ell}$ candidates in in $B_t(s_{k,1})$
	which belong to a $(\mathcal{B}_t,x)$-structure 
	that also contains $v$.
\end{clm}

Before we prove the above claim, let us explain why Property~(R) follows. If we condition on (C) and the good event from Claim~\ref{clm:technical:contain.v}, then we observe the following:
Let any $t\in [2]$, $x\in V_{3-t}$ and any edge sets $Z_1$ and $Z_2$ be given as described in (R) such that 
$|Z_1|=\ln^4(n)$ and $|Z_2|=n^{1/3+\eps/2}$. Then 
the number of vertices in $B_t(s_{k,1})$ for which an edge of 
$Z_1\cup Z_2$ may be relevant with respect to $(\mathcal{B}_t,x)$
is bounded by
$$
|Z_1|\cdot n^{(3^{k+1}-4)\eps} + |Z_2|\cdot n^{1/9+4\eps} = o(|C^x|)\, .
$$
In particular, we then conclude that 
$$
|C^x[Z_1,Z_2]|=(1-o(1))|C^x| > n^{2/3+0.9\eps}\, .
$$
Now, condition on this and notice that for verifying (C) and Claim~\ref{clm:technical:contain.v} we did not need to expose the edges
between $R$ and $B_1(s_{k,1})\cup B_2(s_{k,1})$. Hence, we can expose them now. Then, for each choice of $t,x,Z_1,Z_2$ as described above and each $A\subseteq N(a,R)$ of size $n^{1/3}$ we have that the random variable $e_{G}\big(A,C^x[Z_1,Z_2]\big)$ has distribution 
Bin$(|A||C^x[Z_1,Z_2]|,p)$ with expectation 
\begin{align*}
	\Exp\left(e_{G}\big(A,C^x[Z_1,Z_2]\big)\right)
	= |A| |C^x[Z_1,Z_2]| p
	> n^{1/3+1.9\eps}.  
\end{align*}
Using Chernoff (Lemma~\ref{lem:Chernoff2}) we obtain that the desired inequality $e_{G}\big(A,C^x[Z_1,Z_2]\big) \geq n^{1/3+1.5\eps}$
fails with probability at most $\exp(-n^{1/3+1.8\eps})$.
Hence, taking a union bound over all possible choices of $t$, $x$, $A$, $Z_1$ and $Z_2$, we see that (R) fails to hold with probability at most
\begin{align*}
	& n\cdot \binom{|N(a,R)|}{n^{1/3}}\cdot \binom{n^2}{\ln^{4}(n)}\cdot \binom{n^2}{n^{1/3+\eps/2}}\cdot  e^{-n^{1/3+1.8\eps}} \\
	\leq &
	\exp\left( \left(1+n^{1/3}+2\ln^{4}(n)+2n^{1/3+\eps/2}\right)\ln(n)
	-n^{1/3+1.8\eps}\right) = o(1)\, .
\end{align*}

Thus, it remains to prove Claim~\ref{clm:technical:contain.v}. As its proof is very similar to that of (C), we only provide a sketch here. Let vertices $x$ and $v$ be given as described by Claim~\ref{clm:technical:contain.v}.
We shall prove that for this particular choice of $x$ and $v$, the described property in Claim~\ref{clm:technical:contain.v} fails with probability at most $\exp(-0.3\ln^2(n))$, so that the general statement follows by a union bound over all choices of $x$ and $v$.
By symmetry we may assume that $v\in B_2(s_{\ell,1}^\ast)$.

In order to find the candidates described in Claim~\ref{clm:technical:contain.v}, we may proceed as for the discussion of Property~(C), and find candidate sets $C^x_v(s)$
which are defined as before, just with the difference that from box $B_2(s^\ast_{\ell,1})$ the only vertex we are allowed to use is $v$.
(That is, we replace $B_2(s^\ast_{\ell,1})$ with $\{v\}$.)
In order to bound the sizes of all sets $C^x_v(s)$ we can proceed analogously to the discussion of (C). As hereby most of the failure probabilities appearing in the different steps of the argument are the same as in steps of (C), we do not state these bounds any more, and assume that all likely events happen to hold
at the same time with high probability.

Now, In the analysis of (C), nothing changes until level $\ell$ is reached, since we do not come across the vertex $v$. In particular, it is likely to have
$$
|C^x_v(s)|=|C^x(s)|=n^{(3^{\ell+1}-3)\eps} \ln^{\pm 3^{3\ell}}(n)
$$
for every $s\in L_{\ell}$ as in (C). We next want to upper bound the size of $C^x_v(s_{\ell+1,1})$, i.e.~the candidates from $C^x(s_{\ell+1,1})$
which are adjacent with $v$.
For this we can proceed similarly to the proof of Claim~\ref{claim:nextlevel}.
Let $A_1:=\{v\}$, $A_2:=N(C^x_v(s_{\ell,2}))\cap B_t(s^\ast_{\ell,2})$ and
$A_3:=N(C^x_v(s_{\ell,3}))\cap B_t(s^\ast_{\ell,3})$.
Then $C^x_v(s_{\ell+1,1})$ consists of all vertices of $B_t(s_{\ell+1,1})$ that have a neighbour in each of the sets $A_i$ with 
$i\in [3]$.
We then have that $|A_1|=1$, while Claim~\ref{claim:neighbours} gives that is likely to have
$$|A_i|=\ln^{\pm 2}(n) p|C^x_v(s_{\ell,2})|\cdot |B_t(s^\ast_{\ell,2})|\ .$$ 
Now, condition on the above event. Given any $u\in B_t(s_{\ell+1,1})$ it holds with probability
at most $|A_i|p$ that $u$ has a neighbour in $A_i$ and hence,
$
\Prob(u\in C^x_v(s_{\ell+1,1}))\leq p^3|A_2||A_3|\, .
$
It follows that the random variable $|C^x_v(s_{\ell+1,1})|$
is stochastically dominated by
Bin$(|B(s_{\ell+1,1})|,p^3|A_2||A_3|)$ the expectation of which is
\begin{align*}
	|B_t(s_{\ell+1,1})|\cdot p^3|A_2||A_3|
	=
	\ln^{\pm 4}(n) p^5 |B_t(s_{\ell+1,1})| \prod_{i\in \{2,3\}} |C^x_v(s_{\ell,2})|\cdot |B_t(s^\ast_{\ell,2})| 
	& \leq n^{\eps} p^5 n \prod_{i\in \{2,3\}} \left(n^{(3^{\ell+1}-3)\eps} \cdot n\right) \\
	& = n^{-\frac{1}{3} + \eps\cdot 2\cdot 3^{\ell +1}}\, .
\end{align*}
Hence, applying Chernoff, it follows that with probability at least $1-\exp(-\ln^2(n))$,
$$
|C^x_v(s_{\ell+1,1})| \leq n^{\eps}\cdot \max\left\{1,n^{-\frac{1}{3} + \eps\cdot 2\cdot 3^{\ell +1}} \right\}\, .
$$
Now, we need to consider two cases: $\ell\neq k-1$ and $\ell=k-1$.
In the first case, we obtain
$$
|C^x_v(s_{\ell+1,1})| \leq n^{-2\eps} |C^x(s_{\ell+1,1})|\, .
$$
That is, because of the restriction to $v$, the candidate set
shrinks at least by a factor of $n^{-2\eps}$.
If we now replace 
$C^x(s_{\ell+1,1})$ with $C_v^x(s_{\ell+1,1})$ in the analysis of (C),
it turns out that with high probability the factor $n^{-2\eps}$ carries over to all candidate sets $|C^x_v(s_{j,1})|$ with $j\geq \ell+1$, up to maybe some polylogarithmic factors.
In particular, the relevant number of vertices in level $k$ which we want to estimate can be bounded from above by
$n^{(3^{k+1}-4)\eps}$ as claimed. If otherwise $\ell=k-1$ holds, then we get immediately that
$$
|C^x_v(s_{k,1})| = |C^x_v(s_{\ell+1,1})| \leq n^{\eps -\frac{1}{3} + \eps\cdot 2\cdot 3^{\ell +1}}
\leq n^{\frac{1}{9} + 4\eps}
$$
by our choice of $k$, which concludes the proof.  \hfill \qedhere

\section{Concluding remarks}\label{sec:concluding}

{\bf Making use of local resilience.} In our paper we prove Theorem \ref{thm:main.general}, which states that if $p \geq n^{-2/3 + \eps}$
then playing a $(2:2)$ game on $G\sim G_{n,p}$ Walker a.a.s.~has a strategy to claim a graph that satisfies a given $(p, \eps)$-resilient graph property $\mathcal P$. Hence, by applying other known results on local resilience in random graphs, we can immediately deduce further results on $(2:2)$ Walker-Breaker games on $G_{n,p}$. For instance, when $p \geq n^{-1/2 + \eps}$, then Walker a.a.s.~has strategies to obtain
a pancyclic spanning graph~\cite{krivelevich2010resilient} or the square of an almost spanning cycle~\cite{noever2016local}. Due to a recent result of Fischer et al.~\cite{Fischer2022} we even believe that in the mentioned range Walker can do the square of a Hamilton cycle.

\begin{prob}
	Prove the following: Let $\eps\in (0,1)$. Then, for $p\geq n^{-1/2+\eps}$, playing a $(2:2)$ Walker-Breaker game on the edges of a random graph $G\sim G_{n,p}$, Walker a.a.s.~has a strategy to occupy the square of a Hamilton cycle.
\end{prob}

Note that the strategy given in Section~\ref{sec:strategy} is made in such a way that from some moment on, Walker can play almost like Maker can play in a Maker-Breaker game. Whenever she wants to claim an arbitrary edge at some chosen vertex $v$, she is able to reach that vertex within a constant number of rounds and then claim her desired edge. Having such a strategy at hand, we are able to carry over the argument of Ferber at al.~\cite{F15Gen} to create a random subgraph $H\sim G_{n,q}$ from which Walker a.a.s.~claims all edges except at most an $\eps$-fraction of
edges at every vertex.
Now, in order to apply the result of~\cite{Fischer2022} this argument would not be sufficient. Since Fischer at al.~considered triangle resilience instead of local resilience, Walker would need to ensure to lose at most an $\eps$-fraction of triangles at every vertex. 

\medskip

{\bf Considering different biases.} When we consider the more general $(m:b)$ Connector-Breaker or Walker-Breaker game on $G_{n,p}$, then it turns out that, in contrast to the usual Maker-Breaker setting, the threshold probability for creating a spanning tree or Hamilton cycle highly depends on the given bias $(m:b)$. While most ideas from this paper can be generalized to these doubly biased games, new ideas for Breaker's side are required. This is already a work in progress.

\bigskip

\bibliographystyle{amsplain}
\bibliography{references}

\end{document}